\documentclass[12pt,a4paper]{article}%
\usepackage{amsmath}
\usepackage[left=30mm,top=25mm,right=30mm,bottom=30mm]{geometry}
\usepackage{amsfonts}
\usepackage{amssymb}
\usepackage{amsthm}
\usepackage{color}
\usepackage[utf8]{inputenc}
\usepackage{graphicx}
\usepackage{enumerate}
\usepackage{hyperref}
\providecommand{\U}[1]{\protect\rule{.1in}{.1in}}
\newtheorem{theorem}{Theorem}
\newtheorem{lemma}{Lemma}[section]

\newtheorem{corollary}[lemma]{Corollary}

\newtheorem*{definition*}{Definition}

\newtheorem{problem}[theorem]{Problem}
\newtheorem{proposition}[lemma]{Proposition}

\renewenvironment{proof}[1][Proof]{\noindent\textbf{#1.} }{\ \rule{0.5em}{0.5em}}
\newenvironment{acknowledgement}{\textbf{Acknowledgement.}}{}

\theoremstyle{definition}
\newtheorem*{remark}{Remark}
\newtheorem{example}[lemma]{Example}

\newcommand{\Aut}{\mathrm{Aut}}
\newcommand{\Alt}{\mathrm{Alt}}
\newcommand{\supp}{\mathrm{supp}}
\newcommand{\FAlt}{\mathrm{FAlt}}
\newcommand{\FSym}{\mathrm{FSym}}
\newcommand{\Sub}{\mathrm{Sub}}
\newcommand{\Fix}{\mathrm{Fix}}
\newcommand{\Sh}{\mathrm{Sh}}
\renewcommand{\P}{\mathbb{P}}
\newcommand{\N}{\mathbb{N}}
\newcommand{\Z}{\mathbb{Z}}
\newcommand{\Rst}{\mathrm{Rst}}
\newcommand{\Stab}{\mathrm{Stab}}
\newcommand{\acts}{\curvearrowright}

\AtEndDocument{\bigskip{
  \noindent
  \textsc{Ferenc Bencs, Central European University and MTA Alfr\'ed R\'enyi Institute of Mathematics, Budapest, Hungary} \par  
  \noindent
  \textit{E-mail}: \href{mailto:bencs.ferenc@renyi.mta.hu}{\texttt{bencs.ferenc@renyi.mta.hu}} \par
  \addvspace{\medskipamount}
  
 \noindent
 \textsc{L\'{a}szl\'{o} M\'{a}rton T\'{o}th, Central European University and MTA Alfr\'ed R\'enyi Institute of Mathematics, Budapest, Hungary} \par  
  \noindent
  \textit{E-mail}: \href{mailto:toth.laszlo.marton@renyi.mta.hu}{\texttt{toth.laszlo.marton@renyi.mta.hu}} 
}}

\begin{document}

\title{Invariant random subgroups of groups acting on rooted trees}
\author{Ferenc Bencs\footnote{partially supported by the MTA Rényi Institute Lendület Limits of
Structures Research Group.} \ and L\'{a}szl\'{o} M\'{a}rton T\'{o}th\footnote{supported by the ERC Consolidator Grant 648017. Both partially supported by the Hungarian National Research, Development
and Innovation Office, NKFIH grant K109684.}}
\date{}
\maketitle

\begin{abstract}
We investigate invariant random subgroups in groups acting on rooted trees. Let $\Alt_f(T)$ be the group of finitary even automorphisms of the $d$-ary rooted tree $T$. We prove that a nontrivial ergodic IRS of $\Alt_f(T)$ that acts without fixed points on the boundary of $T$ contains a level stabilizer, in particular it is the random conjugate of a finite index subgroup.  

Applying the technique to branch groups we prove that an ergodic IRS in a finitary regular branch group contains the derived subgroup of a generalized rigid level stabilizer. We also prove that every weakly branch group has continuum many distinct atomless ergodic IRS's. This extends a result of Benli, Grigorchuk and Nagnibeda who exhibit a group of intermediate growth with this property.
\end{abstract}

\section{Introduction} \label{section:intro}

For a countable discrete group $\Gamma$ let $\Sub_{\Gamma}$ denote the compact space of subgroups  $H \leq \Gamma$, with the topology induced by the product topology on $\{0,1\}^{\Gamma}$. The group $\Gamma$ acts on $\Sub_{\Gamma}$ by conjugation. An \emph{invariant random subgroup} (IRS) of $\Gamma$ is a Borel probability measure on $\Sub_{\Gamma}$ that is invariant with respect to the action of $\Gamma$.

Examples include Dirac measures on normal subgroups and uniform random conjugates of finite index subgroups. More generally, for any p.m.p.\ action $\Gamma \acts (X,\mu)$ on a Borel probability space $(X,\mu)$, the stabilizer $\Stab_{\Gamma}(x)$ of a $\mu$-random point $x$ defines an IRS of $\Gamma$. Ab\'ert, Glasner and Vir\'ag  \cite{abert2014kesten} proved that all IRS's of $\Gamma$ can be realized this way. 

A number of recent papers have been studying the IRS's of certain countable discrete groups. Vershik \cite{vershik2012totally} characterized the ergodic IRS's of the group $\FSym(\N)$ of finitary permutations of a countable set. In \cite{abert2012growth} the authors investigate IRS's in lattices of Lie groups. Bowen \cite{bowen2012invariant} and Bowen-Grigorchuk-Kravchenko \cite{bowen2015invariant} showed that there exists a large ``zoo'' of IRS's of non-abelian free  groups and the lamplighter groups $(\Z / p\Z)^n \wr \Z$ respectively. Thomas and Tucker-Drob \cite{thomas2014invariant, thomas2018invariant} classified the ergodic IRS's of strictly diagonal limits of finite symmetric groups and inductive limits of finite alternating groups. Dudko and Medynets \cite{dudko2019invariant} extend this in certain cases to blockdiagonal limits of finite symmetric groups. 

In this paper we study the IRS's of groups of automorphisms of rooted trees. Let $T$ be the infinite $d$-ary rooted tree, and let $\Aut(T)$ denote the group of automorphisms of $T$. 

An \emph{elementary} automorphism applies a permutation to the children of a given vertex, and moves the underlying subtrees accordingly. The group of \emph{finitary} automorphisms $\Aut_f(T)$ is generated by the elementary automorphisms. The \emph{finitary alternating automorphism group} $\Alt_f (T)$  is the one generated by even elementary automorphisms.  

The group $\Aut(T)$ comes together with a natural measure preserving action. The \emph{boundary} of $T$ -- denoted $\partial T$ -- is the space of infinite rays of $T$. It is a compact metric space with a continuous $\Aut(T)$ action and an ergodic invariant measure $\mu_{\partial T}$. 

For some natural classes of groups IRS's tend to behave like normal subgroups. In \cite{abert2012growth} the Margulis Normal Subgroup Theorem  is extended to IRS's, it is shown that every nontrivial ergodic IRS of a lattice in a higher rank simple Lie group is a random conjugate of a finite index subgroup. On the other hand, the finitary alternating permutation group $\FAlt(\N)$ is simple, in particular it has no finite index subgroups, but as Vershik shows in \cite{vershik2012totally} it admits continuum many ergodic IRS's. 

The group $\Alt_f(T)$ is an interesting mixture of these two worlds. It is both locally finite and residually finite, and all its nontrivial normal subgroups are level stabilizers. The Margulis Normal Subgroup Theorem does not extend to IRS's, as the stabilizer of a random boundary point gives an infinite index ergodic IRS.  However, once we restrict our attention to IRS's without fixed points, the picture changes. 

\begin{theorem} \label{theorem:alternating_finitary_ffp}
Let $H$ be a fixed point free ergodic IRS of $\Alt_f (T)$, with $d \geq 5$. Then $H$ is the uniform random conjugate of a finite index subgroup. In other words $H$ contains a level stabilizer.
\end{theorem}

Note, that an IRS $H$ is \emph{fixed point free} if it has no fixed points on $\partial T$ almost surely. In general let $\Fix(H)$ denote the closed subset of fixed points of $H$ on $\partial T$.

When we do not assume fixed point freeness IRS's of $\Alt_f(T)$ start to behave like the ones in $\FAlt(\N)$. In the case of $\FAlt(\N)$, any nontrivial ergodic IRS contains a specific (random) subgroup that arises by partitioning the base space in an invariant random way and then taking the direct sum of deterministic subgroups on the parts. We proceed to define a (random) subgroup of $\Alt_f(T)$ which highly resembles these subgroups. For a group $\Gamma$ acting on a set $X$ the \emph{pointwise stabilizer} of a subset $C \subseteq X$ is denoted by $\Stab_{\Gamma}(C)=\{\gamma \in \Gamma \mid \gamma c= c, \  \forall c \in C \}$.
\begin{figure}[h]
   \centering
   \includegraphics[width=12.5cm]{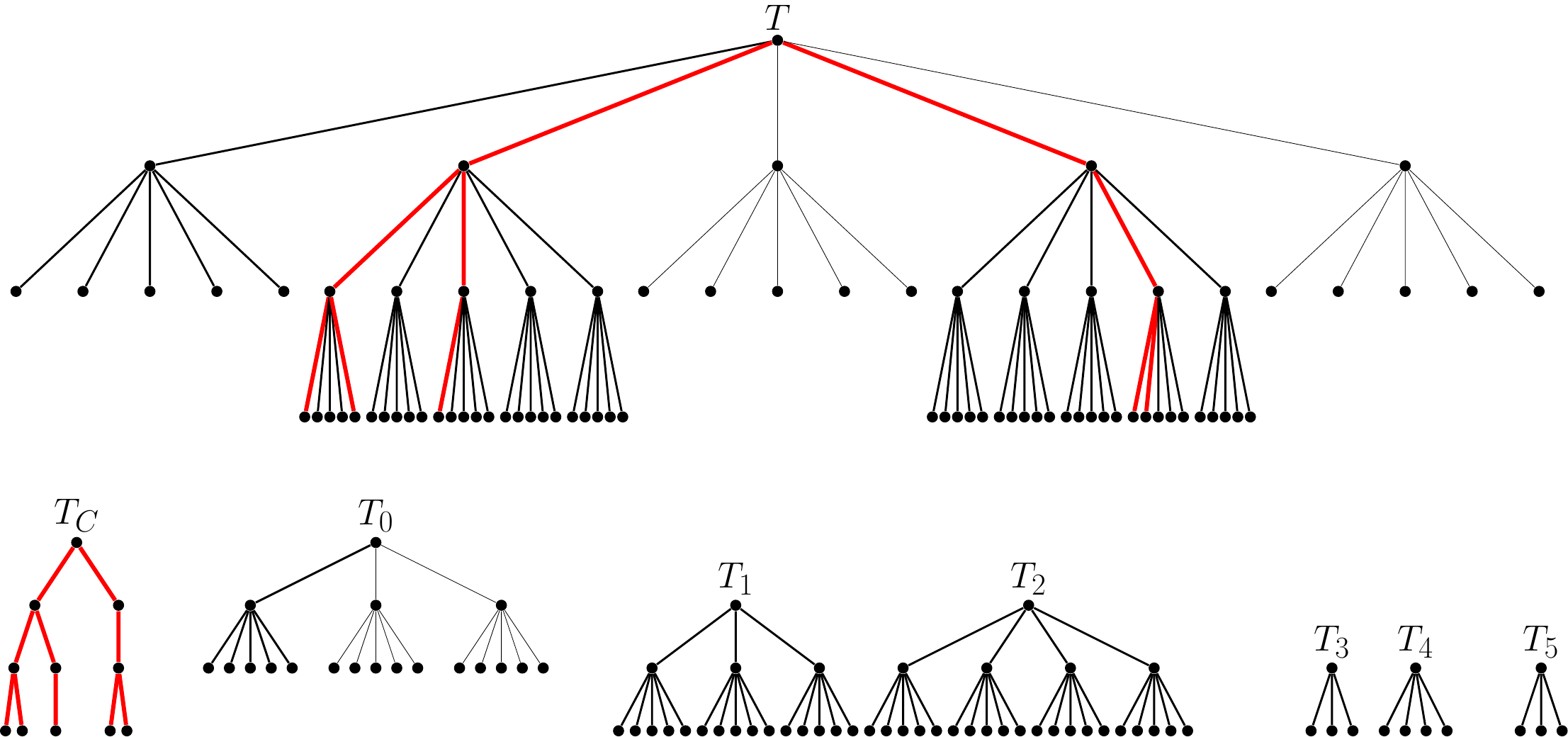}
   \caption{Decomposition of $T$ with respect to $C$}
   \label{fig:intro_trees}
\end{figure}

Every closed subset $C \subseteq \partial T$ corresponds to a rooted subtree $T_C$ with no leaves. The complement of $T_C$ in $T$ is a union of subtrees $T_0, T_1, \ldots$ as in Figure \ref{fig:intro_trees}. Choose an integer $m_i$ for each $T_i$, and let $\mathcal{L}_{m_i}(T_i)$ stand for the $m_i^{\mathrm{th}}$ level of the tree $T_i$. We define $L\big(C, (m_i)\big)$ to be the direct sum of level stabilizers in the $T_i$:

\[L\big(C, (m_i)\big)= \bigoplus_{i \in \N} \Stab_{\Alt_f(T)}\big(\mathcal{L}_{m_i}(T_i)\big).\]

It is easy to see that $\Fix\big(L\big(C, (m_i)\big)\big)=C$. We call such an $L\big(C, (m_i)\big)$ a \emph{generalized congruence subgroup} with respect to the fixed point set $C$.\\
\indent
Let $\widetilde{C}$ be the translate of $C$ with a Haar-random element from the compact group $\Alt(T) = \overline{\Alt_f (T)}$. Then $L\big(\widetilde{C},(m_i)\big)$ becomes an ergodic IRS of $\Alt_f(T)$ with fixed point set $\widetilde{C}$. 

\begin{theorem} \label{theorem:finitary_alternating}
Let $H$ be an ergodic IRS of $\Alt_f(T)$, with $d \geq 5$. Then $\Fix(H)$ is the Haar-random translate of a fixed closed subset $C$. Moreover, there exists $(m_i)$ such that the generalized congruence subgroup $L\big(\Fix(H),(m_i)\big)$ is contained in $H$ almost surely. 
\end{theorem}

We can exploit our methods to prove new results on branch groups as well. We postpone the formal definition of branch groups to Section \ref{section:preliminaries}. The examples to keep in mind are the groups $\Aut_f(T)$, $\Alt_f (T)$ and groups defined by finite automata, such as the first Grigorchuk group $\mathfrak{G}$. 

In \cite{dudko2018diagonal} Dudko and Grigorchuk show that branch groups admit infinitely many distinct atomless (continuous) ergodic IRS's. In the ergodic case being atomless means that the measure is not supported on a finite set.
In \cite{benli2015universal} Benli, Grigorchuk and Nagnibeda exhibit a group of intermediate growth $U_{\Lambda}$  with continuum many distinct atomless ergodic IRS's. 
We are able to find continuum many such IRS's in weakly branch groups in general.

\begin{theorem}\label{theorem:continuum_many}
Every weakly branch group admits continuum many distinct atomless ergodic IRS's.
\end{theorem}

Note that the universal Grigorchuk group $U_{\Lambda}$ in \cite{benli2015universal} is not weakly branch, as it is not transitive on the levels. Nevertheless, by its construction it factors onto branch groups, which by Theorem \ref{theorem:continuum_many} have continuum many distinct atomless ergodic IRS's, and those can be lifted to distinct IRS's of $U_{\Lambda}$. Thus Theorem \ref{theorem:continuum_many} gives an alternate proof of the main result of \cite{benli2015universal}.

The main focus of this paper is to understand IRS's in branch groups via their action on the boundary. While Theorem \ref{theorem:continuum_many} is easy to state and covers a wide class of groups, in fact it is the byproduct of understanding $\Fix(H)$ along the way towards the more technical Theorem \ref{theorem:finitary}.

A key ingredient in Theorems \ref{theorem:alternating_finitary_ffp} and  \ref{theorem:finitary_alternating}
is to analyze the orbit-closures of IRS's on $\partial T$. For any subgroup $L \leq \Aut(T)$ taking the closures of orbits of $L$ gives an equivalence relation on $\partial T$, that is $L$ acts minimally on each class. It turns out that nontrivial orbit-closures of IRS's are necessarily clopen. 

\begin{theorem}\label{theorem:clopen}
Let $H$ be an ergodic IRS of a countable regular branch group $\Gamma$. Then almost surely all orbit-closures of $H$ on $\partial T$ that are not fixed points are clopen. In particular if $H$ is fixed point free, then $H$ has finitely many orbit-closures on $\partial T$ almost surely.
\end{theorem}

In a group $\Gamma$ the \emph{rigid stabilizer} of a vertex $v \in V(T)$ is the set $\Rst_{\Gamma}(v) \subseteq \Gamma$ of automorphisms that fix all vertices except the descendants of $v$. The rigid stabilizer of the level $\mathcal{L}_n$ is 

\[\Rst_{\Gamma}(\mathcal{L}_n) = \prod_{v \in \mathcal{L}_n} \Rst_{\Gamma}(v).\]

In \cite[Theorem 4]{grigorchuk2000just} Grigorchuk showed that nontrivial normal subgroups in branch groups contain the derived subgroup  $\Rst'_{\Gamma}\big(\mathcal{L}_m (T)\big)$ for some $m \in \N$. Our next theorem can be thought of as a generalization of this statement for finitary regular branch groups.

Using the decomposition of $T$ with respect to $C$ above we can define a \emph{generalized rigid level stabilizer} $L(C, m_i)$ by taking the direct sum of the rigid level stabilizers $\Rst_{\Gamma}\big(\mathcal{L}_{m_i}(T_i)\big)$ instead of the $\Stab_{\Gamma}\big(\mathcal{L}_{m_i}(T_i)\big)$ we used before. Throughout this paper $L(C, m_i)$ always stands for this generalized rigid level stabilizer, it just so happens that in the case of $\Alt_f(T_d)$ above the rigid stabilizers are in fact the stabilizers.  The next theorem generalizes Theorem \ref{theorem:alternating_finitary_ffp} and Theorem \ref{theorem:finitary_alternating} for \emph{finitary} regular branch groups.

\begin{theorem} \label{theorem:finitary}
Let $\Gamma$ be a finitary regular branch group, and let $H$ be a nontrivial ergodic IRS of $\Gamma$. Then  $\Fix(H)$ is the Haar-random translate of a closed subset $C$ with an element from $\overline{\Gamma}$. Also there exists $(m_i)$  such that $H$ almost surely contains the derived subgroup $L'\big(\Fix(H),(m_i)\big)$ of a generalized rigid level stabilizer. In particular if $H$ is fixed point free, then $H$ almost surely contains $\Rst'_{\Gamma}\big(\mathcal{L}_m(T)\big)$ for some $m \in \N$.
\end{theorem}

Already in the case of $\Aut_f(T)$ with $d=2$ the abelianization of $\Aut_f(T)$ equals $(\Z / 2\Z)^{\N}$. This itself gives rise to a lot of IRS's, which makes the following consequence of Theorem \ref{theorem:finitary} somewhat surprising. 

\begin{theorem}\label{theorem:atomic}
All ergodic fixed point free IRS's in finitary regular branch groups are supported on finitely many subgroups, and therefore are the uniform random conjugates of a subgroup with finite index normalizer.
\end{theorem}

One can think of Theorem \ref{theorem:atomic} as a dual of Theorem \ref{theorem:continuum_many}. Also note that merely containing $\Rst'_{\Gamma}\big(\mathcal{L}_m(T)\big)$ does not imply finite index normalizer.

Note that the assumption of finitary elements in these results is quite restrictive, but our approach makes it a necessary technical condition. In the Grigorchuk group $\mathfrak{G}$ the elements are not finitary. In this case our methods yield a weaker result on the closures of IRS's.

\begin{theorem} \label{theorem:branch_group_closure}
Let $\Gamma$ be a countable regular branch group, and let $H$ be a nontrivial ergodic $IRS$ of $\Gamma$. Then there exists $(m_i)$ such that $\overline{H}$ contains the derived subgroup $L'\big(\Fix(H), (m_i) \big)$ of a generalized rigid level stabilizer almost surely, where the elements of the rigid stabilizers in $L\big(\Fix(H), (m_i) \big)$ can be chosen from $\overline{\Gamma}$ instead of $\Gamma$.
\end{theorem} 

However, classifying IRS's of the discrete Grigorchuk group $\mathfrak{G}$ is still open.

\begin{problem} \label{problem:grigorchuk_group}
What are the (fixed point free) ergodic IRS's of the first Grigorchuk group $\mathfrak{G}$? Is it true, that a fixed point free ergodic IRS of $\mathfrak{G}$ contains a congruence subgroup almost surely?
\end{problem}

\begin{remark}
After the current paper was made available on arXiv the main results were extended to countable, non-finitary branch groups by Zheng in \cite{zheng2019rigid}. Her work settled Problem \ref{problem:grigorchuk_group} as well. Without claiming precedence we mention that the case of $\mathfrak{G}$ can also be settled using our Theorem \ref{theorem:branch_group_closure} and the work of Pervova in \cite{pervova2000everywhere}. Pervova proves that $\mathfrak{G}$ has no proper subgroup that is dense in $\overline{\mathfrak{G}}$. Using this one can show that if an $H \leq \mathfrak{G}$ is dense in some rigid stabilizer, then $H$ actually contains it. Note that this argument is specific to $\mathfrak{G}$, while Zheng's work is general. We thank Rostislav Grigorchuk for bringing Pervova's result to our attention.
\end{remark}

The structure of the paper is as follows. We introduce the basic notions of the paper in Section \ref{section:preliminaries} and state some lemmas leading towards Theorem \ref{theorem:clopen}. In Section \ref{section:orbit_closures} we investigate the actions of IRS's on the boundary and prove Theorems \ref{theorem:continuum_many} and \ref{theorem:clopen}. Section \ref{section:main_theorem} is dedicated to understanding the structure of IRS's in finitary regular branch groups and proving Theorem \ref{theorem:finitary}. We show how Theorems \ref{theorem:atomic} and \ref{theorem:branch_group_closure} follow from our earlier results in Section \ref{section:corollaries}. In the Appendix we prove a few technical details that we postpone during the exposition.

\medskip

\begin{acknowledgement}
The authors would like to thank Mikl\'os Ab\'ert for introducing them to some of these questions and for helpful remarks along the way. We would also like to thank Tatiana Smirnova-Nagnibeda and Rostislav Grigorchuk for helpful remarks on a previous version of this paper. We are thankful to the anonymous referee for all the constructive remarks and suggestions for improving this paper.
\end{acknowledgement}

\section{Preliminaries} \label{section:preliminaries}

In this section we introduce the basic notions discussed in the paper. Notation mostly follows \cite{bartholdi2003branch}, which we recommend as an introduction to automorphisms of rooted trees and branch groups.

\subsection{Automorphisms of rooted trees}

\label{subsec:aut_of_rooted_trees}

Let $T$ be a locally finite tree rooted at $o$, and let $d_T$ denote the graph distance on $T$. For any vertex $v$ the \emph{parent} of $v$ is the unique neighbor $u$ of $v$ with $d_T(u,o)=d_T(v,o)-1$. Accordingly, the \emph{children} of $u$ are all the neighbors $v$ of $u$ with $d_T(v,o) = d_T(u,o)+1$. Similarly we use the phrases \emph{ancestors} and \emph{descendants} of a vertex $v$ to refer to vertices that can be reached from $v$ by taking some number of steps towards or away from the root respectively. The $n^{\mathrm{th}}$ level of $T$ is the set of vertices $\mathcal{L}_n = \{v \in V(T) \mid d_T(v,o)=n\}$. 

To effectively talk about automorphisms of a rooted tree $T$ one has to distinguish the vertices. For any vertex $v$ we fix an ordering of the children of $v$. In the case of the $d$-ary tree this corresponds to thinking of $T$ as the set of finite length words $Y^*$ over the alphabet $Y$ with $d$ letters. The empty word represents the root, and the parent of any word $w_1 w_2\ldots w_n$ is $w_1 w_2 \ldots w_{n-1}$. Being an ancestor of $v$ corresponds to being a prefix of the word corresponding to $v$.

An automorphism of $T$ (which preserves the root) corresponds to a permutation of the words which preserves the prefix relation. For an element $\gamma \in \Aut(T)$ and a word $w \in Y^*$ we denote by $w^{\gamma}$ the image of $w$ under $\gamma$. For a letter $y \in Y$ we have $(wy)^{\gamma} = w^{\gamma}y'$ where $y'$ is a uniquely determined letter in $Y$. The map $y \mapsto y'$ is a permutation of $Y$, we refer to it as the \emph{vertex permutation} of $\gamma$ at $w$ and denote it $(w)\gamma$. 

Considering all the vertex permutations $\big((w)\gamma\big)_{w \in Y^*}$ gives us the \emph{portrait} of $\gamma$, which is a decoration of the vertices of $T$ with elements from the symmetric group $S_d$. In turn any assignment of these vertex permutations -- that is, every possible portrait -- gives an automorphism of $T$. Note that one has to perform these vertex permutations ``from bottom to top''.

An automorphism $\gamma$ is finitary, if it has finitely many nontrivial vertex permutations. It is alternating, if all are from the alternating group $A_d$.

Let $S_d^{\textrm{wr}(n)}$ denote the $n$-times iterated permutational wreath product of the symmetric group $S_d$. That is,  let $[d]=\{1, \ldots, d\}$ and set 
\[S_d^{\textrm{wr}(n)} = \underbrace{( ( S_d \wr_{[d]}\ldots)\wr_{[d]}S_d) \wr_{[d]} S_d}_{n}.\]

Then $S_d^{\textrm{wr}(n)}$ is isomorphic to the automorphism group of the $d$-ary rooted tree of depth $n$. These groups can be embedded in $\Aut(T)$ as acting on the first $n$ levels. The group $\Aut_f$ is the union of these embedded finite groups. The full automorphism group $\Aut(T)$ however is isomorphic to the projective limit $\varprojlim S_d^{\textrm{wr}(n)}$ with the projections being the natural restrictions of the permutations. 

The groups $\Alt_f(T)$ and $\Alt (T)$ are in a similar relationship with the finite groups $A_d^{\textrm{wr}(n)}$.

\subsection{The boundary of $T$}

The \emph{boundary} of $T$ is the set of infinite simple paths starting from $o$, and is denoted $\partial T$. For two distinct paths $p_1=(u_0, u_1, \ldots)$ and $p_2=(v_0, v_1, \ldots)$ with $u_n, v_n \in \mathcal{L}_n$ their distance is defined to be 
\[d_{\partial T}(p_1, p_2) = \frac{1}{2^k}, \textrm{ where } k = \max\{n \mid u_n=v_n\}.\]
Two infinite paths are close it they have a long common initial segment. This distance turns $\partial T$ into a compact, totally disconnected metric space. 

The \emph{shadow} of $v$ on $\partial T$, denoted by $\Sh(v)$ is the set of paths passing through $v$. Similarly the shadow of $v$ on $\mathcal{L}_n$ is the set $\Sh_{\mathcal{L}_n}(v)$ of descendants of $v$ in $\mathcal{L}_n$. The sets $\Sh(v)$ form a basis for the topology of $\partial T$. Define the probability measure $\mu_{\partial T}$ by setting its value on this basis: \[\mu_{\partial T} \big(\Sh(v)\big) = \frac{1}{d^n} \textrm{ for every } v \in \mathcal{L}_n.\]

A $\mu_{\partial T}$-random point of $\partial T$ is a random infinite word $(w_1w_2\ldots)$ with each letter chosen uniformly from the set $Y$. 

As $\gamma \in \Aut(T)$ permutes the vertices, it induces a bijection on $\partial T$, so we have an action of $\Aut(T)$ on $\partial T$. This action is by isometries and preserves the measure $\mu_{\partial T}$.  

The objects in relation of the tree considered in this paper include vertices $v \in V(T)$, points $x \in \partial T$, closed subsets $C \subseteq \partial T$ and later 3-colorings of the vertices $\varphi: V(T) \to \{r,g,b\}$. For any such object $z$ let $z^{\gamma}$ denote its translate by $\gamma$.

\subsection{Topology on $\Aut(T)$}

We equip $\Aut(T)$ with the topology of pointwise convergence. This can be metrized by the following distance:

\[d_{\Aut(T)} ( \gamma_1, \gamma_2) = \frac{1}{2^k}, \textrm{ where } k=\max\{n \mid \gamma_1\vert_{\mathcal{L}_n} = \gamma_2\vert_{\mathcal{L}_n}\}.\]
Two automorphisms are close if they act the same way on a deep level of $T$. This metric turns $\Aut(T)$ into a compact, totally disconnected group. 

The action $\Aut(T) \acts \partial T$ is continuous in the first coordinate as well. For any subgroup $H \leq \Aut(T)$ the set $\Fix(H)$ is closed in $\partial T$, and similarly for any set $C \subseteq \partial T$ its pointwise stabilizer $\Stab_{\Aut(T)} (C)$ is closed in $\Aut(T)$.

For a subgroup $\Gamma \leq \Aut(T)$ its closure $\overline{\Gamma}$ is a closed subgroup of $\Aut(T)$, and therefore it is compact. We note that $\overline{\Aut_f (T)} = \Aut (T)$ and $\overline{\Alt_f (T)} = \Alt (T)$. Even though the groups $\Gamma$ we are considering are discrete, their closures in $\Aut(T)$ always carry a unique Haar probability measure. 

For any object $z$ in relation to the tree we write $\widetilde{z}$ for its Haar random translate, that is $z^{\gamma}$ where $\gamma \in \overline{\Gamma}$ is chosen randomly according to the Haar measure.

\subsection{Fixed points and orbit-closures in $\partial T$}

We aim to understand the IRS's of $\Gamma$ through their actions on $\partial T$. The first step is to look at the set of fixed points. The boundary $(\partial T, d_{\partial T} )$ is a compact metric space, so let $(\mathcal{C}, d_H)$ denote the compact space of closed subsets of $\partial T$ with the Hausdorff metric.

\begin{lemma} \label{lemma:measurable_fixedpoint_set}
The map $H \mapsto \Fix(H)$ is a measurable and $\Gamma$-equivariant map from $\Sub_{\Gamma}$ to $(\mathcal{C}, d_H)$.
\end{lemma}

Equivariance is trivial, while the proof of the measurability is a standard argument. We postpone it to the Appendix. 

Lemma \ref{lemma:measurable_fixedpoint_set} implies that the fixed points of the IRS constitute a $\Gamma$-invariant random closed subset of $\partial T$. 
We will also consider the orbit-closures of the subgroup on $\partial T$. For a subgroup $H \leq \Aut(T)$ let $\mathcal{O}_H$ denote the set of orbit-closures of the action $H \acts \partial T$. It is easy to see that $\mathcal{O}_{H}$ is a partition of $\partial T$ into closed subsets. Note that all fixed points are orbit-closures. Denote by $\mathcal{O}$ the space of all possible orbit-closure partitions on $\partial T$, i.e. $\mathcal{O}=\{\mathcal{O}_H~|~H\le \Aut(T)\}$. This $\mathcal{O}$ is a subset of all the possible partitions of $\partial T$.

As earlier, we would like to argue that the map $H\mapsto \mathcal{O}_H$ is a measurable map, with respect to the appropriate measurable structure on $\mathcal{O}$. This allows us to associate to our IRS a $\Gamma$-invariant random partition (into closed subsets) of $\partial T$. We will then analyze these invariant random objects on the boundary.  

To this end we introduce a metric on the space $\mathcal{O}$. Denote by $\mathcal{O}_{H,\mathcal{L}_n}$ the partition of $\mathcal{L}_n$ into $H$-orbits. As $\mathcal{L}_n$ is finite, there is no need to take closure here.  

\begin{definition*}
 Let $P=\mathcal{O}_H\in\mathcal{O}$ be the orbit-closure partition of $H$ and $n\in \mathbb{N}$. Then let $P_n$ be the orbit-structure of $H$ on $\mathcal{L}_n$, i.e. \[P_n=\mathcal{O}_{H,\mathcal{L}_n}.\]
 
 For $P \neq Q\in\mathcal{O}$ let 
 \[
  d_{\mathcal{O}}(P,Q)=\min_{n\in \mathbb{N}}\left\{\frac{1}{2^n}~\big|~P_n=Q_n\right\}.
 \]
\end{definition*}
Observe that if $P_n=Q_n$, then $P_{n-1}=Q_{n-1}$, so the above distance measures how deep one has to go in the tree to see that two partitions are distinct. This definition turns $(\mathcal{O},d_{\mathcal{O}})$ into a metric space. To check that distinct points cannot have zero distance we argue that if $x=(v_0,v_1,\ldots)$ and $y=(u_0, u_1, \ldots)$ are two rays such that $v_n$ and $u_n$ are in the same orbit in $\mathcal{L}_n$ for all $n$, then $y$ is indeed in the closure of the orbit of $x$.  

The group $\Aut(T)$ acts on $\mathcal{O}$ in a natural way by shifting the sets of the partition. The resulting partition is again in $\mathcal{O}$ because $(\mathcal{O}_H)^{\gamma} = \mathcal{O}_{H^{\gamma}}$ for $\gamma \in \Aut(T)$.

\begin{lemma} \label{lemma:measurable_orbit_closures}
 The map $H\mapsto \mathcal{O}_H$ is measurable and $\Gamma$-equivariant.
\end{lemma}

Again, equivariance is obvious, and measurability is proved in the Appendix.

\subsection{Invariant random objects on $\partial T$}

Now we study invariant random closed subsets and partitions on the boundary. We show that the invariance can be extended to $\overline{\Gamma}$, which carries a Haar measure. Ergodic objects turn out to be random translates according to this Haar measure.

\begin{lemma} \label{lemma:closure_invariant}
Every $\Gamma$-invariant random closed subset of $\partial T$ is in fact $\overline{\Gamma}$-invariant. Similarly a $\Gamma$-invariant random $P\subseteq \mathcal{O}$ is $\overline{\Gamma}$-invariant.
\end{lemma}

\begin{proof}
Let $P(\mathcal{C})$ denote the set of probability measures on $\mathcal{C}$. The action of $\overline{\Gamma}$ on $\partial T$ gives rise to a translation action on $(\mathcal{C}, d_H)$, which in turn gives rise to an action on $P(\mathcal{C})$. 

We claim that this action $\overline{\Gamma} \times P(\mathcal{C}) \to P(\mathcal{C})$ is continuous in both coordinates with respect to the pointwise convergence topology on $\overline{\Gamma}$ and the weak star topology on $P(\mathcal{C})$. 

The weak topology on $P(\mathcal{C})$ is metrizable by the Lévy - Prokhorov metric, which is defined as follows:
\[\pi(\mu,\nu)= \inf \{\varepsilon >0 \mid \mu(A) \leq \nu(A^{\varepsilon}) + \varepsilon,\textrm{ and } \nu(A) \leq \mu(A^{\varepsilon}) + \varepsilon \textrm{ for all } A \subseteq \mathcal{C} \textrm{ Borel} \}.\] 
Here $A^{\varepsilon}$ denotes the set elements of $\mathcal{C}$ with $d_H$ distance at most $\varepsilon$ from $A$.

If $\gamma_1$ and $\gamma_2$ agree on the first $n$ levels of $T$, then for every $x \in \partial T$ we have $d(\gamma_1 x, \gamma_2 x) \leq 1/2^n$. This implies, that for a compact set $C \in \mathcal{C}$ we have \[d_H(\gamma_1 C, \gamma_2 C) \leq 1/2^n.\] This in turn implies that for all $A \subseteq \mathcal{C}$ Borel we have $\gamma_1 A \subseteq \gamma_2 A^{1/2^n}$ and vice versa.

This means, that $\big((\gamma_1)_{*}\mu\big) (A) =\mu( \gamma_1^{-1} A) \leq \mu (\gamma_2^{-1} A^{1/2^n})= \big( (\gamma_2)_{*}\mu \big) (A^{1/2^n})$, so as a consequence $\pi\big((\gamma_1)_{*}\mu, (\gamma_2)_{*}\mu \big) \leq 1/2^n$. That is, the action is continuous in the first coordinate.

Continuity in the second coordinate is an easy exercise, as it turns out that the elements of $\overline{\Gamma}$ act by isometries on $(\partial T,d_{\partial T})$, $(\mathcal{C}, d_H)$ and $\big(P(\mathcal{C}), \pi\big)$ respectively. 

As $\Gamma$ is a dense subset of $\overline{\Gamma}$, by continuity we can say that if some $\mu \in P(\mathcal{C})$ is $\Gamma$-invariant then it is also $\overline{\Gamma}$-invariant, thus proving the statement for invariant random closed subsets.

The proof for invariant random partitions follows the exact same steps after substituting $(\mathcal{C},d_H)$ with $(\mathcal{O}, d_{\mathcal{O}})$ everywhere. 
\end{proof}

\begin{remark}
The fact that the same lemma holds with the same proof for closed subsets and partitions is not a coincidence. A closed subset $C$ can be thought of as a partition into the two sets $C$ and $C^c$ (the complement might not be closed). While it is not generally true that this partition is in $\mathcal{O}$ -- it might not arise as an orbit-closure partition of some $H \subseteq \Aut(T)$ -- but it still can be approximated on the finite levels. Indeed define $C_n$ to be the set of vertices $v$ on $\mathcal{L}_n$ with $\Sh(v) \cap C \neq \emptyset$. Then $\Sh(C_n)= \bigcup_{v \in C_n} \Sh(v)$ is the $(1/2^n)$-neighborhoods of $C$ in $\partial T$. Consequently, for closed subsets $C, D \subseteq \partial T$, their Hausdorff distance is 
\[d_H(C,D) = \min_{n \in \N }\left\{ \frac{1}{2^n} \mid C_n=D_n \right\},\]
which coincides with the definition of $d_{\mathcal{O}}$.   
\end{remark}

\begin{lemma}\label{lemma:unique}
Any ergodic $\overline{\Gamma}$-invariant random closed subset of $\partial T$ is the $\gamma$ translate of a fixed closed subset $C$, where $\gamma \in \overline{\Gamma}$ is a uniform random element chosen according to the Haar measure. Similarly an ergodic  $\overline{\Gamma}$-invariant random partition from $\mathcal{O}$ is the Haar-random translate of some fixed $P \in \mathcal{O}$.
\end{lemma}

\begin{proof}
We introduce an equivalence relation on closed subsets of $\partial T$: we say that $C_1 \sim C_2$ if and only if there is an automorphism $\gamma \in \overline{\Gamma}$ such that $C_1^{\gamma}=C_2$. Let $[C]$ denote the equivalence class of $C$.

Define the following metric on equivalence classes that measures how well one can overlap two arbitrary sets from the classes: \[d([C_1], [C_2]) = \min_{\gamma \in \overline{\Gamma} }\{d_H(C_1^{\gamma}, C_2)\}.\]

The minimum exists by compactness of $\overline{\Gamma}$, and standard arguments using the fact that $\overline{\Gamma}$ acts by isometries on $(\mathcal{C}, d_H)$ show that this is well defined and indeed a metric. 

The function $C \rightarrow [C]$ is measurable (in fact continuous) and $\overline{\Gamma}$-invariant, hence it is almost surely a constant by the ergodicity of the measure.

In other words the measure is concentrated on one equivalence class, say $[C]$. However $[C]$ is a homogeneous space of $\overline{\Gamma}$, i.e. the action of $\overline{\Gamma}$ on $[C]$ is isomorphic to the action $\overline{\Gamma} \acts \overline{\Gamma}/\Stab_{\overline{\Gamma}}(C)$. $\Stab_{\overline{\Gamma}}(C)$ is a closed and therefore compact subgroup of $\overline{\Gamma}$, and as such $\overline{\Gamma}/\Stab_{\overline{\Gamma}}(C)$ carries a unique invariant measure. Of course picking a random translate of $C$ is an invariant measure, so the two must coincide.    

The result for partitions again follows word for word, by writing $P$ for $C$ and $(\mathcal{O},d_{\mathcal{O}})$ for $(\mathcal{C}, d_H)$ everywhere. 
\end{proof}

\begin{remark}
A way to put the previous lemmas into a general framework is the following: let $G$ be a metrizable compact group acting continuously on a compact metric space $(X,d)$ and let $\Gamma\leq G$ be a dense subgroup. Then any $\Gamma$-invariant measure on $X$ is also $G$-invariant. Moreover if the metric $d$ is $G$-compatible, then any ergodic $\Gamma$-invariant measure on $X$ has the distribution of a random $G$-translate of a fixed element in $X$.
\end{remark}

If the IRS $H \leq \Gamma$ is ergodic, then so is the associated invariant random closed subset. This means that $\Fix(H)$ is the random translate of a fixed closed subset $C$. Similarly $\mathcal{O}_H$ is the random translate of some partition $P$.

\subsection{Branch Groups}

For a vertex $v$ of $T$ let $T_v$ denote the induced subtree of $T$ on $v$ and its descendants. We denote by $\Stab_{\Gamma}(v)$ the stabilizer of $v$ in $\Gamma$. Every $\gamma \in \Stab_{\Gamma}(v)$ acts on $T_v$ by an automorphism, which we denote $\gamma_v$. Then $U_v = \{\gamma_v \mid \gamma \in \Stab_{\Gamma}(v)\}$ is a subgroup of $\Aut(T_v)$. $U_v$ is the group of automorphisms of $T_v$ that are realized by some element of $\Gamma$. 

The trees we are considering are regular, so $T_v$ is canonically isomorphic to $T$. (The isomorphism preserves the ordering of the vertices on each level. If we think of the vertices as finite words over a fixed alphabet, then this isomorphism just deletes the initial segment of each word in $T_v$.) This identification of the trees allows us to compare the action of $G$ on $T$ to the action of $U_v$ on $T_v$. In particular we say that $\Gamma$ is a \emph{self-similar} group, if $U_v$ is equal to $G$ for all $v \in V(T)$ (under the above identification of the trees they act on).

For a vertex $v \in V(T)$ let $\Rst_{\Gamma}(v)$ denote the \emph{rigid stabilizer} of $v$, that is the subgroup of elements of $\Gamma$ that fix every vertex except the descendants of $v$. Clearly $\Rst_{\Gamma}(v) \leq \Stab_{\Gamma}(v)$. For a subset of vertices $V \subseteq \mathcal{L}_n$ the rigid stabilizer of the set is $\Rst_{\Gamma}(V) = \prod_{v \in V} \Rst_{\Gamma}(v)$.

Throughout the paper we will be able to prove statements in varying levels of generality, so we introduce several notions of branching. In all cases we assume $\Gamma$ to be transitive on all levels. We say that $\Gamma$ is \emph{weakly branch}, if all rigid vertex stabilizers $\Rst_{\Gamma}(v)$ are nontrivial. We say that the group $\Gamma$ is \emph{branch}, if for all $n$ the rigid level stabilizer $\Rst_{\Gamma}(\mathcal{L}_n)$ is a finite index subgroup of $\Gamma$. Finally we define \emph{regular branch groups}. 

\begin{definition*}
Suppose the self-similar group $\Gamma$ has a finite index subgroup $K$. The group $K^d$ is a subgroup of $\Aut(T)$, each component acting independently on $T_{v_i}$ where $\{v_1, \ldots, v_d\}$ are the vertices on $\mathcal{L}_1$. We say that $\Gamma$ is a \emph{regular branch group} over $K$, if $K$ contains $K^d$ as a finite index subgroup.
\end{definition*}

In self-similar groups the action on any subtree $T_v$ is the same as on $T$, however for some $v_1, v_2 \in \mathcal{L}_n$ we might not be able to move $T_{v_1}$ and $T_{v_2}$ independently. This independence (up to finite index) is what we required in the definition above. The following lemma is straightforward and we leave the proof to the reader.

\begin{lemma} Having finite index and being a direct product remains to be true after taking closures:
\begin{enumerate}
\item Let $K \subseteq \Gamma$ be a subgroup of finite index. Then $\overline{K}$ is a finite index subgroup of $\overline{\Gamma}$;
\item  $\overline{\underbrace{K \times \dots \times K}_{d}} = \underbrace{\overline{K} \times \dots \times \overline{K}}_{d}$.\end{enumerate}
\end{lemma}

\section{Fixed points and orbit-closures of IRS's} \label{section:orbit_closures}

In this section we prove Theorems \ref{theorem:continuum_many} and \ref{theorem:clopen}.
To a closed subset $C$ on the boundary one can associate two natural subgroup of $\Gamma$, the pointwise stabilizer of $C$ and the setwise stabilizer of $C$. The pointwise stabilizer gives us a big ``zoo'' of IRS's when we choose $C$ as a $\overline{\Gamma}$-invariant closed subset, proving Theorem \ref{theorem:continuum_many}. The setwise stabilizer will play a key role in the proof of Theorem \ref{theorem:clopen}.

In order to investigate these stabilizers we introduce a coloring to encode $C$ on the tree $T$. The coloring will help analyzing the Haar random translate $\widetilde{C}$.

\subsection{Closed subsets of the boundary}
\label{subsection:closed_subsets_of_boundary}

To every closed subset of the boundary $C \subseteq \partial T$ we associate a vertex coloring $\varphi:V(T) \to \{r,g,b\}$ with 3 colors: red, green and blue. If a vertex has its shadow completely in $C$, then color it red. If it has its entire shadow in the complement of $C$, then color it blue. Otherwise color it green. 

\[\varphi(v) = \left\{ \begin{array}{ll} r, & \textrm{if } \Sh(v) \subseteq C; \\ b, & \textrm{if }  \Sh(v) \cap C = \emptyset;   \\ g, & \textrm{otherwise}. \end{array} \right.\]

All descendants of a red vertex are red, and similarly all descendants of blue vertices are blue. On the other hand all ancestors of a green vertex are green.

$C$ being clopen is equivalent to saying that after some level all vertices are either red or blue. So if $C$ is not clopen, then there are green vertices on all the levels. Using K\"onig's lemma we see that there is an infinite ray with vertices colored green. This ray corresponds to a boundary point of $C$. As the complement of $C$ is open, we get that every vertex on this infinite ray has a blue descendant.

\begin{lemma} \label{lemma:infinite_orbit}
Let $\Gamma \subseteq \Aut(T)$ be a group of automorphisms that is transitive on every level. Let $\varphi:V(T) \to \{r,g,b\}$ be a vertex coloring with the colors red, green and blue, and suppose it satisfies the above properties, namely:

\begin{enumerate}
\item \label{property:red_blue} descendants of red and blue vertices are red and blue respectively;
\item ancestors of green vertices are green; (This formally follows from \ref{property:red_blue}.)
\item there is an infinite ray $(u_0, u_1, \ldots)$ of green vertices such that for each $u_i$ there exists some descendant of $u_i$ which is blue. 
\end{enumerate}
 
Then $\varphi$ has infinitely many $\Gamma$-translates.
\end{lemma}

\begin{proof}
The root $u_0$ is colored green. It has a blue descendant, say on the level $n_1$. We denote this blue descendant $w_{n_1}$. By the transitivity assumption there is some $\gamma_1 \in \Gamma$, such that $\gamma_1(w_{n_1})=u_{n_1}$. Furthermore, $u_{n_1}$ has a blue descendant, on some level $n_2$, we denote it $w_{n_2}$. We choose $\gamma_2 \in \Gamma$ such that $\gamma_2(w_{n_2})=u_{n_2}$, and so on. One can easily check, that moving $\varphi$ with the different $\gamma_i$ yields different colorings. Indeed $\varphi^{\gamma_{i}} (u_{n_j}) = g$ for all $j <i$, and $\varphi^{\gamma_{i}} (u_{n_i}) = b$, and this shows that the $\varphi^{\gamma_i}$ are pairwise distinct. See Figure \ref{fig:colored_ray}.
\begin{figure}[h]
   \centering
   \includegraphics[width=12cm]{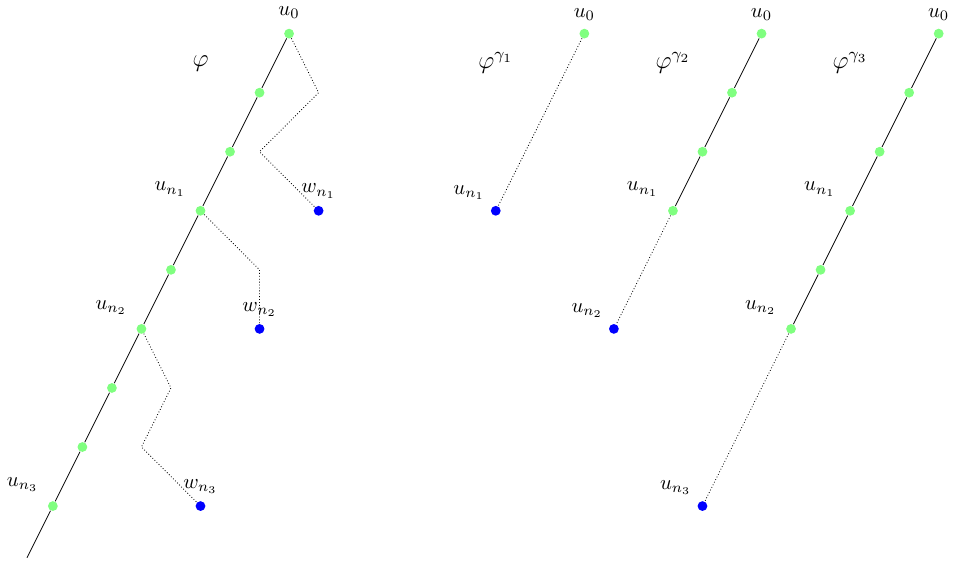}
   \caption{Distinct colorings}
   \label{fig:colored_ray}
\end{figure}
\end{proof}

\begin{corollary} \label{corollary:atomless_coloring}
Let $\Gamma$ and $\varphi$ be as in Lemma $\ref{lemma:infinite_orbit}$. Then the uniform (Haar) random $\overline{\Gamma}$-translate of $\varphi$ is an atomless measure on the space of all 3-vertex-colorings.
\end{corollary}

\begin{proof}
If there was some translate $\varphi^{g}$, $g \in \overline{\Gamma}$ which occurred with positive probability, then all its $\Gamma$-translates would occur with the same positive probability. Furthermore $\varphi^{g}$ would also satisfy the assumptions of Lemma \ref{lemma:infinite_orbit}, which then implies that it has infinitely many $\Gamma$-translates, and they would have infinite total measure, which is a contradiction. 
\end{proof}

\begin{corollary}\label{corollary:atomless_IRCS}
If $C$ is not clopen, then its random $\overline{\Gamma}$-translate $\widetilde{C}$ is an atomless measure on $(\mathcal{C}, d_H)$.
\end{corollary}

\subsection{Continuum many distinct atomless ergodic IRS's in \\weakly branch groups}

We are already equipped to prove Theorem \ref{theorem:continuum_many}. The contents of this subsection are not necessary for the proof of Theorem \ref{theorem:clopen}.

\begin{proof}[Proof of Theorem \ref{theorem:continuum_many}]
We argue that for any closed subset $C \subseteq \partial T$ the random subgroup $\Stab_{\Gamma}(\widetilde{C})$ is an ergodic IRS. This follows from $\widetilde{C}$ being an ergodic invariant random closed subset. 

We also claim that if $[C_1] \neq [C_2]$, then the corresponding IRS's are distinct. To prove this we first observe that in weakly branch groups taking the stabilizer $\Stab_{\Gamma}(C)$ of a closed subset $C$, and then looking at the fixed points of that subset we get back $C$.

\begin{lemma} \label{lemma:fix_stab_id}
For any $C \subseteq \partial T$ closed we have $\Fix\big(\Stab_{\Gamma}(C)\big) = C$.
\end{lemma}
\begin{proof}
The key idea -- present in \cite[Proposition 8]{benli2015universal} and earlier works credited there -- is to show, that for any $x \notin C$, with $x=(u_0, u_1, \ldots)$ we can find some $n$ large enough such that $\Sh(u_n) \cap  C=\emptyset$, and some $\gamma \in \Rst_{\Gamma}(u_n)$ with $x^{\gamma} \neq x$. 

Indeed such an $n$ exists as the complement of $C$ is open. By weak branching there exists some $\gamma_0 \in \Rst_{\Gamma}(u_n)$ moving some descendant of $u$ denoted $v$ to $v'\neq v$ on $\mathcal{L}_m$, $m \geq n$. By transitivity we can find some $\eta \in \Stab_{\Gamma}(u_n)$ with $v = u_m^{\eta}$. Now $u_m^{\eta \gamma_0 \eta^{-1}} = (v')^{\eta^{-1}} \neq u_m$, so $\gamma = \eta \gamma_0 \eta^{-1} \in \Rst_{\Gamma} (u_n)$, and $x^{\gamma} \neq x$ as witnessed on $\mathcal{L}_m$. 

As $\Sh(u_n) \cap  C=\emptyset$ we have $ \Rst_{\Gamma}(u_n) \subseteq \Stab_{\Gamma}(C)$. The existence of $\gamma$ shows that $x \notin \Fix\big(\Stab_{\Gamma}(C)\big)$, which implies $\Fix\big(\Stab_{\Gamma}(C)\big) \subseteq C$, which is the nontrivial inclusion.      
\end{proof}

\medskip
To show that $[C_1] \neq [C_2]$ implies that $\Stab_{\Gamma}(\widetilde{C_1})$ and $\Stab_{\Gamma}(\widetilde{C_2})$ distinct simply consider the function $H \mapsto [\Fix(H)]$ on ergodic IRS's. Using Lemma \ref{lemma:fix_stab_id} we have

\[\big[\Fix\big(\Stab_{\Gamma}(\widetilde{C_1})\big)\big] = [\widetilde{C_1}] = [C_1].\]

This implies that the constructed IRS are distinct. By Corollary \ref{corollary:atomless_IRCS} we know that if $C$ is not clopen then $\widetilde{C}$ is atomless. Then Lemma \ref{lemma:fix_stab_id} implies that $\Stab_{\Gamma}{\widetilde{C}}$ is an atomless IRS.

There are continuum many non-$\overline{\Gamma}$-equivalent closed (but not clopen) subsets of $\partial T$, as one can construct a closed subset $C_r$ with $\mu_{\partial T} (C_r) = r$ for any $r \in [0,1]$, and if $r$ is irrational then $C$ is not clopen. 
\end{proof}

\subsection{Random colorings in regular branch groups}

To build towards Theorem \ref{theorem:clopen}, we resume the investigation of random colorings in the case when the ambient group is regular branch.

Let $\Gamma$ be a regular branch group over $K$. Consider the finite index subgroup $\overline{K^{d^n}} \leq \Stab_{\overline{\Gamma}}(\mathcal{L}_n) = \cap_{v \in \mathcal{L}_n} \Stab_{\overline{\Gamma}}(v)$, and let $\{t_1, \ldots, t_l \}$ be a transversal to $\overline{K^{d^n}}$ in $\overline{\Gamma}$. We can think of a random element $\gamma$ of $\overline{\Gamma}$ as $\gamma = \gamma_0 \cdot k$, where $\gamma_0$ is chosen uniformly from the transversal and $k$ is chosen according to the Haar measure on $\overline{K^{d^n}}$. 

Take $\varphi$ to be a 3-vertex-coloring as in Lemma \ref{lemma:infinite_orbit}, and fix an infinite green ray $(u_0, u_1, \dots)$ with all vertices $u_i$ having blue descendants. Let $\widetilde{\varphi} = \varphi^{\gamma}$ denote the translate of $\varphi$ by the Haar random group element $\gamma$. Conditioning on $\gamma_0 = t_i$ we get a conditional distribution $(\widetilde{\varphi} \vert \gamma_0=t_i)$. Note that this random coloring is always the same up to the $n^{\textrm{th}}$ level, and $t_i$ already determines where the random translate of $(u_0, u_1, \dots)$ intersects $\mathcal{L}_n$, namely at $v=u_n^{t_i}$.

\begin{lemma} \label{lemma:branch_atomless_coloring}
The restriction of the random coloring $(\widetilde{\varphi} \vert \gamma_0=t_i)$ to $T_v$ is atomless. 
\end{lemma}
\begin{proof}
Let $\varphi^{t_i}|_{T_v}$ denote the restriction of the coloring $\varphi^{t_i}$ to $T_v$. Since $\varphi^{t_i}|_{T_v}$ satisfies the assumptions of Lemma \ref{lemma:infinite_orbit}, its $\Gamma$-orbit is infinite. As $K$ is finite index in $\Gamma$, its $K$ orbit is also infinite. The random coloring $(\widetilde{\varphi} \vert \gamma_0=t_i)|_{T_v}$ is a translate of $\varphi^{t_i}|_{T_v}$ by a random element of $\overline{K}$ seen to be acting on $T_v$. As the $K$ orbit of $\varphi^{t_i}|_{T_v}$ is infinite, the random coloring $(\widetilde{\varphi} \vert \gamma_0=t_i)|_{T_v}$ is atomless.
\end{proof}

\begin{lemma} \label{lemma:branch_zero_probability}
Fix any isomorphism $f:V(T_v) \to V(T_{v'})$ between $T_v$ and $T_{v'}$ for some $v' \in \mathcal{L}_n$. Then the probability that $f$ respects the colorings we get by restricting $(\widetilde{\varphi} \vert \gamma_0=t_i)$ to $T_v$ and $T_{v'}$ respectively is 0. 

\[\mathbb{P} \Big[ \big((\widetilde{\varphi} \vert \gamma_0=t_i)\vert_{T_v}\big)^f = (\widetilde{\varphi} \vert \gamma_0=t_i)\vert_{T_{v'}} \Big]=0.\] 
\end{lemma}

\begin{remark}
In the formula above, all the randomness comes from the choice of $(\widetilde{\varphi} \vert \gamma_0=t_i)$. First, one generates a random instance of $(\widetilde{\varphi} \vert \gamma_0=t_i)$ and gets a coloring of $T$. Second, one  restricts this coloring to the two subtrees. Third, one checks whether $f$ maps one restriction to the other. We claim that this ($f$ mapping one restriction to the other) happens with probability zero.
\end{remark}

\begin{proof}
The restricted colorings $\varphi^{t_i} \vert_{T_v}$ and $\varphi^{t_i} \vert_{T_{v'}}$ are translated by the random elements $k_1, k_2 \in \overline{K}$ respectively. These $k_1$ and $k_2$ are independent since they are two coordinates of a Haar random element from $\overline{K}^{d^n}$. Furthermore we know from Lemma \ref{lemma:branch_atomless_coloring} that $(\varphi^{t_i} \vert_{T_v})^{k_1}$ is atomless, and hence $\big((\varphi^{t_i} \vert_{T_v})^{k_1}\big)^f$ is also atomless. This together with the independence of $k_1$ and $k_2$ implies that

\[ \mathbb{P}\Big[  \big((\varphi^{t_i} \vert_{T_v})^{k_1}\big)^f = (\varphi^{t_i} \vert_{T_{v'}})^{k_2}  \Big] = 0. \]
\end{proof}

\subsection{Proof of Theorem \ref{theorem:clopen}}

The idea of the proof is to show that the \emph{setwise} stabilizer of a Haar random translate $\widetilde{C}$ of a closed but not clopen subset $C$ has a fixed point in $\widetilde{C}$. With some considerations one can apply this to the orbit-closures of $H$, which are setwise stabilized by $H$. 

\begin{proposition} \label{proposition:fixedpointfree_stabilizer}
Let $\Gamma$ be a countable regular branch group over $K$. Suppose $C$ is a closed subset of $\partial T$, and as before $\widetilde{C}$ denotes the uniform $\overline{\Gamma}$-translate of $C$. Let $L \leq \Gamma$ denote the setwise stabilizer of $\widetilde{C}$ in $\Gamma$. If $C$ is not clopen, then $L$ has a fixed point in $\widetilde{C}$ almost surely.
\end{proposition}

\begin{remark}
As both $L$ and $\widetilde{C}$ are random objects, there might be ambiguity in what one means by $L$ having a fixed point in $\widetilde{C}$. Note that $\widetilde{C}$ is generated randomly, but $L$ is defined deterministically once $\widetilde{C}$ is chosen. The statement is to be understood for almost all random instances of $\widetilde{C}$.
\end{remark}

\begin{proof} Associate the coloring $\varphi: V(T) \to \{r,g,b\}$ to $C$ as before: vertices with shadows contained in $C$ are colored red, vertices with shadows in the complement are colored blue, everything else is colored green. As automorphisms move the set $C$ the coloring moves with it. 

Choose a point $x_0 \in \partial T$ which is on the boundary of $C$, that is $x_0 \in C \setminus \mathrm{int}(C)$. Being a boundary point means that every vertex on the path $(u_0, u_1, \ldots)$ corresponding to $x_0$ is green, and we can find a blue vertex among the descendants of $u_i$ for all $i$.

Let $\widetilde{C}=\gamma C$, $\widetilde{\varphi}=\varphi^{\gamma}$ and $\widetilde{x}_0=x_0^{\gamma}$, where $\gamma \in \overline{\Gamma}$ is a uniform random element. While $\widetilde{C}$, $\widetilde{\varphi}$ and $\widetilde{x}_0$ are random objects, they are strongly dependent as they are obtained using the same $\gamma$.

Fix an element $\eta \in \Gamma$. We will study the probability that $\eta$ stabilizes $\widetilde{C}$ and does not fix $\widetilde{x}_0$, and conclude that it is 0. If $\eta$ stabilizes $\widetilde{C}$ setwise then it preserves $\widetilde{\varphi}$. 

First we present our argument in the finitary case to illustrate the key idea, then explain how to deal with the non-finitary case.

\vspace{1 mm}
\noindent {\bf Finitary case.} Assume $\eta$ is finitary, that is we can find a level $n$ with vertices $\mathcal{L}_n= \{v_1, \ldots, v_{d^n}\}$ such that $\eta$ moves the subtrees $T_{v_i}$ hanging off the $n^{\textrm{th}}$ level rigidly. The condition that $\widetilde{x}_0$ is moved has to be witnessed on $\mathcal{L}_n$. Assume $v_1, \ldots, v_l$ are moved by $\eta$ and $v_{l+1}, \ldots, v_{d^n}$ are fixed.

Let us assume that $\widetilde{\varphi}$ is preserved by $\eta$, and the ray corresponding to $\widetilde{x}_0$ is moved by $\eta$. Then $\widetilde{x}_0 \cap \mathcal{L}_n = v_i$ for some $i \leq l$ with $v_j = \eta v_i \neq v_i$. We condition on $v_i$ and claim that 
\[
\mathbb{P}\big[(\widetilde{\varphi}\vert_{T_{v_i}})^{\eta} = \widetilde{\varphi}\vert_{T_{v_j}} ~\big|~ \widetilde{x}_0 \cap \mathcal{L}_n = v_i \big] = 0.
\]


In the case of $\Gamma = \Alt_f (T)$ this is an easy consequence of Corollary \ref{corollary:atomless_coloring}, because the two restrictions are independent.  
There are finitely many choices of $v_i$, so the probability of moving $\widetilde{x}_0$ while stabilizing $\widetilde{C}$ is 0.

In the general case when $\Gamma$ is a regular branch group we condition on $\gamma_0 =t_i$ as in Lemma \ref{lemma:branch_zero_probability}, and with $f$ the canonical isomorphism between $T_{v_i}$ and $T_{v_j}$ we conclude that the conditional probability of $\eta$ preserving $(\widetilde{\varphi} \vert \gamma_0=t_i)$ is 0. There are finitely many choices for $t_i$, so again we conclude that the probability of moving $\widetilde{x}_0$ while stabilizing $\widetilde{C}$ is 0. 

As $\Gamma$ is countable this means that with probability 1 the whole setwise stabilizer of $\widetilde{C}$ fixes $\widetilde{x}_0$.

\vspace{1 mm}
\noindent {\bf Non-finitary case.} When $\eta$ is not finitary there are two points where the previous argument has to be adapted:
\begin{enumerate}[i)]
\item \label{problem:rigid} the trees $T_{v_i}$ are not moved rigidly;
\item \label{problem:witness} the $n^{\textrm{th}}$ level might not witness that $\widetilde{x}_0$ is moved by $\eta$. 
\end{enumerate}

Notice however that \ref{problem:rigid}) is not a real problem. The group element $\eta$ is fixed, and it induces an isomorphism $f=\eta \vert_{T_{v_i} \to T_{v_j}}$. We can refer to Lemma \ref{lemma:branch_zero_probability} as we did in the finitary case.  In fact this is why we needed to allow an arbitrary (but fixed) $f$ in Lemma \ref{lemma:branch_zero_probability}.

To work our way around \ref{problem:witness}) we notice that 
\begin{eqnarray}
\mathbb{P}\big[\widetilde{x}_0^{\eta} \neq \widetilde{x}_0, \textrm{ but } \widetilde{\varphi}^{\eta}=\widetilde{\varphi}\big] & \leq & \mathbb{P}[\widetilde{x}_0^{\eta} \neq \widetilde{x}_0 \textrm{ and } (\widetilde{x}_0 \cap \mathcal{L}_n)^{\eta} = (\widetilde{x}_0 \cap \mathcal{L}_n)] +  \nonumber \\
&  & + \ \mathbb{P}\big[(\widetilde{x}_0 \cap \mathcal{L}_n)^{\eta} \neq (\widetilde{x}_0 \cap \mathcal{L}_n) \textrm{ and } \widetilde{\varphi}^{\eta}=\widetilde{\varphi}\big] \label{eqn:prob_bound}.
\end{eqnarray}

The argument we presented in the finitary case show that 
\[\mathbb{P}\big[\widetilde{\varphi}^{\eta}=\widetilde{\varphi} ~\big|~ (\widetilde{x}_0 \cap \mathcal{L}_n)^{\eta} \neq (\widetilde{x}_0 \cap \mathcal{L}_n)\big] =0,\] 
so the second term in (\ref{eqn:prob_bound}) is 0 for all $n \in \N$.

On the other hand the first term in (\ref{eqn:prob_bound}), the probability of $\widetilde{x}_0$ being moved by $\eta$ but this not being witnessed on $\mathcal{L}_n$ tends to 0 as $n \to \infty$. Indeed, $\Fix(\eta) = \bigcap_{n \in \N} \Sh \big(\Fix_{\mathcal{L}_n}(\eta)\big)$, and this intersection is decreasing. Therefore \[\mathbb{P}[\widetilde{x}_0^{\eta} \neq \widetilde{x}_0 \textrm{ and } (\widetilde{x}_0 \cap \mathcal{L}_n)^{\eta} = (\widetilde{x}_0 \cap \mathcal{L}_n)]= \mathbb{P}\big[\widetilde{x}_0 \in \Sh\big(\Fix_{\mathcal{L}_n}(\eta)\big) \setminus \Fix(\eta) \big] \to 0.\] 

Consequently we get $\mathbb{P}\big[\widetilde{x}_0^{\eta} \neq \widetilde{x}_0, \textrm{ but } \widetilde{\varphi}^{\eta}=\widetilde{\varphi}\big] =0$. This finishes  the proof because $\Gamma$ is countable.
\end{proof}

\medskip

\begin{proof}[Proof of Theorem \ref{theorem:clopen}]
 By ergodicity and Lemma \ref{lemma:unique} we know that there exists a $P\in\mathcal{O}$, such that $\widetilde{P}$ has the same distribution as $\mathcal{O}_H$. Let us choose a closed set $C$ which is not a single point from the partition $P$. We aim to use Proposition \ref{proposition:fixedpointfree_stabilizer} to conclude that $C$ is clopen. For that we will couple $H$ and $\widetilde{C}$ such that $H \leq L$ holds almost surely, where $L$ is the setwise stabilizer IRS of $\widetilde{C}$. Then $H$ moving all points of $C$ implies the same for $L$, which then through Proposition \ref{proposition:fixedpointfree_stabilizer} implies that $C$ is clopen.

Let $X=(P,C)\in\mathcal{O}\times\mathcal{C}$. Consider the diagonal action of $\overline{\Gamma}$ on $\mathcal{O}\times\mathcal{C}$. Let $\widetilde{X}$ be the Haar random translate of $X$. This way we obtained that the first coordinate of $\widetilde{X}$ has the same distribution as $\mathcal{O}_H$, the second coordinate has the same distribution as $\widetilde{C}$, and the second coordinate is always a closed subset in the partition given by the first coordinate. 

Now we use the transfer theorem (see Theorem 6.10. of \cite{kallenberg2002foundations}) to obtain a random element $C_H$ of $\mathcal{C}$, such that 
$(\mathcal{O}_H,C_H)\stackrel{d}{=}\widetilde{X}$. (Recall that $Y \stackrel{d}{=} Z$ stands for the two random variables $Y$ and $Z$ having the same distribution.) The first coordinate of $\widetilde{X}$ always contains the second, therefore $C_H\in\mathcal{O}_H$ and clearly $C_H\stackrel{d}{=}\widetilde{C}$. Choosing $L$ to be the setwise stabilizer of $C_H$ concludes the proof.
\end{proof}

\section{IRS's in regular branch groups} \label{section:main_theorem}

Our goal is to understand all ergodic IRS's $H$ of $\Gamma$. Let $\widetilde{C} = \Fix(H)$. Lemmas \ref{lemma:closure_invariant} and \ref{lemma:unique} tell us that $\widetilde{C}$ is the $\gamma$ translate of a fixed closed subset $C \subseteq \partial T$, where $\gamma \in \overline{\Gamma}$ is Haar random. First we exhibit some concrete examples which are worth to keep in mind and to motivate the decomposition of the tree in Subsection \ref{subsection:decomposition}. We study the action of $H$ on the parts in Subsection \ref{subsection:action_of_H}. The last two subsections contain the proof of the main theorem of the paper.

\subsection{Examples}
We show a few examples to keep in mind. For simplicity let $d=5$, and $\Gamma= \Alt_f(T)$. Recall that in this group the normal subgroups are the level stabilizers $\Stab_{\Gamma}(\mathcal{L}_n)$, and the quotients are the finite groups $A_d^{\mathrm{wr}(n)}$. Recall that $A_d^{\mathrm{wr}(n)}$ stands for the $n$-fold iterated permutational wreath product of $A_d$, see Subsection \ref{subsec:aut_of_rooted_trees}.

\begin{example} \label{example:fixed_point_free_IRS}
Pick $n \in \N$, and a finite subgroup $L \leq A_d^{\mathrm{wr}(n)}$. Let $\widetilde{L}$ be the uniform random conjugate of $L$ in $A_d^{\mathrm{wr}(n)}$, and $H$ be the preimage of $\widetilde{L}$ under the quotient map, that is $H= \widetilde{L} \cdot \Stab_{\Gamma}(\mathcal{L}_n)$. Then $H$ is an ergodic fixed point free IRS of $\Gamma$. Note that this construction also works if $G$ is only eventually $d$-ary, i.e.\ vertices on the first few levels might have different number of children.
\end{example}

Theorem \ref{theorem:alternating_finitary_ffp} states that all ergodic fixed point free IRS of $\Alt_f(T)$ are listed in Example \ref{example:fixed_point_free_IRS}. We give a very broad outline of the proof for this case in the hope that it makes the subsequent proof of the stronger Theorem \ref{theorem:finitary} more transparent and motivates Proposition \ref{proposition:few_exceptions} that we state beforehand.

\medskip
\begin{proof}[Outline of proof of Theorem \ref{theorem:alternating_finitary_ffp}]
By Theorem \ref{theorem:clopen} we know that an ergodic fixed point free IRS $H$ has finitely many clopen orbit-closures on the boundary. A deep enough level $\mathcal{L}_{k_0}$ witnesses this partition into clopen sets, that is parts in this partition are of the form $\bigcup_{v \in I} \Sh(v)$ where $I \subseteq \mathcal{L}_{k_0}$. Moreover, $H$ acts transitively on the different parts on each $\mathcal{L}_n$ with $n \geq k_0$. 

This means that we can find some elements $\gamma_1, \ldots, \gamma_l \in \Alt_f(T)$ such that $\mathbb{P}\big[\{\gamma_1, \ldots, \gamma_l\} \subseteq H \big]>0$, and (conditioned on $\{\gamma_1, \ldots, \gamma_l\} \subseteq H$) the $\gamma_i$ generate the $H$-orbits on $\mathcal{L}_{k_0}$. Since $\Alt_f(T)$ is finitary, we can choose some ($k \geq k_0$) such that the $\gamma_i$ have no vertex permutations below $\mathcal{L}_k$. 

We then study $\mathbb{P}\big[\{\gamma_1, \ldots, \gamma_l\} \subseteq H' \big]>0$ for an ergodic IRS $H'$ in the finite group $A_d^{\mathrm{wr}(n)}$, with $n \geq k$. Then $H'$ is the uniform random conjugate of some fixed subgroup $L \leq A_d^{\mathrm{wr}(n)}$, and $L$ has to contain many conjugates of the set $\{\gamma_1, \ldots, \gamma_l\}$. One can show that, provided $n$ is sufficiently large, this implies that $L$ contains a whole level stabilizer $\Stab_{A_d^{\mathrm{wr}(n)}}(\mathcal{L}_m)$ for some $m\geq k$, where $m$ does not depend on the choice of $n$.

Using that $\Alt_f(T)$ is the union of the $A_d^{\mathrm{wr}(n)}$, together with some additional analysis of ergodic components, one can show that actually $H$ contains $\Stab_{\Alt_f(T)}(\mathcal{L}_m)$ almost surely. 
\end{proof}

\begin{example} \label{example:general_IRS}
Pick a random point $x \in \partial T$, this will be the single fixed point of the IRS $H$. Deleting the edges of the ray $(u_0, u_1, \ldots)$ corresponding to $x$ from $T$ we get infinitely many disjoint trees, where the roots $u_n$ have degree $4$, while the rest of the vertices have $5$ children. Pick any fixed point free IRS for each of these trees as in example \ref{example:fixed_point_free_IRS}, randomize them independently and take their direct sum to be $H$. This construction works with other random fixed point sets instead of a single point as well. 
\end{example}

\begin{example} \label{example:coupling}
A modification of the previous example is the following. Let $x \in \partial T$ be random as before, and do the exact same thing for all the trees hanging of the ray $(u_0, u_1, \ldots)$ except for the first two, $T_1$ and $T_2$ rooted at $u_0$ and $u_1$ respectively. The finitary alternating automorphism groups of these trees are $\Alt_f(T) \wr A_4$. Now pick an (ergodic) fixed point free IRS of the finitary alternating and bi-root-preserving automorphism group of $T_1 \cup T_2$, which is $(\Alt_f(T) \wr A_4) \times (\Alt_f(T) \wr A_4)$, and use this to randomize $H$ on $T_1 \cup T_2$. We will show that this is different from the previous examples. When we pick an IRS of $(\Alt_f(T) \wr A_4) \times (\Alt_f(T) \wr A_4)$ we pick some $n \in \N$, assume that the stabilizers of the $n^{\mathrm{th}}$ levels in $T_1$ and $T_2$ are in the IRS, and pick a random conjugate of some $L \leq (\Alt_f(T) \wr A_4) \times (\Alt_f(T) \wr A_4)$ to extend the stabilizer. If we pick for example $L=\big\{(\gamma, \gamma) \mid \gamma \in (\Alt_f(T) \wr A_4) \}$, then the IRS we construct will not be the direct product of IRS's on the two components, because the ``top'' parts of the subgroups are coupled together. Taking a random conjugate of $L$ makes the coupling random as well, but nonetheless in every realization of $H$ there is some nontrivial dependence between the actions of $H$ on $T_1$ and $T_2$. 
\end{example}

\subsection{Decomposition of $T$} \label{subsection:decomposition}

To the set of fixed points $\widetilde{C}$ we associate a subtree $T_{\widetilde{C}}$,  the union of all the rays corresponding to the points of $\widetilde{C}$. Note that $\widetilde{C}$ is random (as it depends on $H$), but the construction of $T_{\widetilde{C}}$, given $\widetilde{C}$, is deterministic. So both $\widetilde{C}$ and $T_{\widetilde{C}}$ are random objects, but they are both deterministic given $H$.

All elements of $H$ fix all vertices of the tree $T_{\widetilde{C}}$, so understanding $H$ requires us to focus on the rest of $T$. We will decompose $T$ according to the subtree $T_{\widetilde{C}}$. Note that the following decomposition is slightly different to the one in the introduction as it is easier to work with.

On $\mathcal{L}_n$ denote the set of fixed vertices $F_n =V( T_{\widetilde{C}}) \cap \mathcal{L}_n$. Remove all edges $E(T_{\widetilde{C}})$ from $T$, the remaining graph $T'$ is a union of trees.  

\begin{figure}[h]
   \centering
   \includegraphics[width=12.5cm]{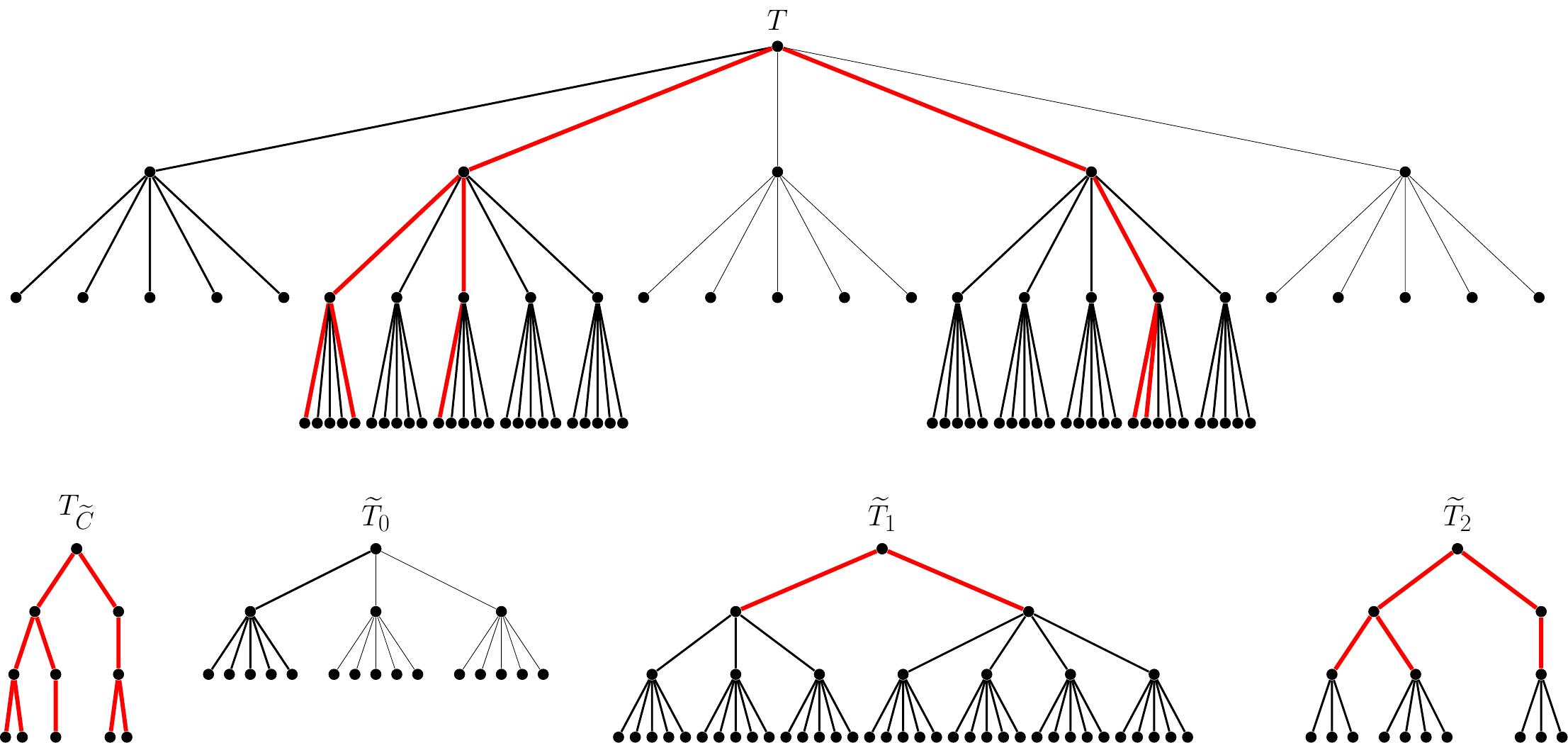}
   \caption{Decomposition of $T$ with respect to $\widetilde{C}$}
   \label{fig:trees}
\end{figure}

Let $\widetilde{T}_0$ be the connected component of $T'$ containing the root of $T$. In other words it is the tree starting at the root in $T'$. In general let $\widetilde{T}_n$ be constructed as follows. The first $n$ levels on $\widetilde{T}_n$ will be the same as the first $n$ levels of $T_{\widetilde{C}}$, and beyond that select the connected components of $T'$ containing the vertices of $F_n$. The vertices of $\widetilde{T}_n$ are exactly the vertices of $T$ that can be reached from the root by taking $n$ steps in $T_{\widetilde{C}}$ and then some number of steps in $T'$. See Figure \ref{fig:trees}. 

The boundary $\partial T$ decomposes as well. Clearly $\partial T_{\widetilde{C}} = \widetilde{C}$, and \[\partial T = \widetilde{C} \cup \partial \widetilde{T}_0 \cup \partial \widetilde{T}_1 \cup \ldots\]

Each $\partial \widetilde{T}_i$ is $H$-invariant, and a clopen and therefore compact subset of $\partial T$. It is the union of clopen orbit-closures from $\mathcal{O}_H$ because of Theorem \ref{theorem:clopen}, so it is the union of fintely many. 

In the remaining part of this section we will prove the following. Fix an equivalence class $[C']$ that appears in $\mathcal{O}_H$ with probability 1.

\begin{remark}
Recall that $[C']$ denotes the equivalence class of a closed $C' \subseteq \partial T$ under the equivalence relation $C_1 \sim C_2 \Leftrightarrow \exists~\gamma \in \overline{\Gamma},~ C_1^{\gamma} = C_2$. The orbit-closure partition $\mathcal{O}_H$ is random, but the equivalence classes of orbit-closures appearing in $\mathcal{O}_H$ are the same fixed collection for almost every realization of $\mathcal{O}_H$, because $\mathcal{O}_H$ is ergodic.
\end{remark}

We claim that there exists $m^* \in \N$ such that for almost every $H$ the following holds: for every $C \in \mathcal{O}_H$, $C \sim C'$ there exists a subset $C_{m^*} \subseteq \mathcal{L}_{m^*}$ with $\Sh(C_{m^*}) = C$ such that $\Rst_{\Gamma}'(C_{m^*}) \leq H$. Note that $m^*$ depends on $[C']$, but not on the realization of $H$. 



Applying the claim above for the finitely many orbit-closures $C$ that constitute $\partial \widetilde{T}_i$ and taking a maximum of the $m^*$-s yields that for some $m_i \geq i$ we have $\Rst_{\Gamma}'\big(\mathcal{L}_{m_i}(\widetilde{T}_i)\big) \subseteq H$. Knowing this for all $i$ yields
\[\bigoplus_{i \in \N} \Rst_{\Gamma}'\big(\mathcal{L}_{m_i}(\widetilde{T}_i)\big) \subseteq H,\]
which is equivalent to the statement of Theorem \ref{theorem:finitary}.

\subsection{The action of $H$ on the $\widetilde{T}_i$} \label{subsection:action_of_H}

Before we turn to proving Theorem \ref{theorem:finitary} we argue that all IRS's resemble the previous examples in the sense that their projections on the $\widetilde{T}_n$ are fixed point free IRS's in $\Stab_{\Gamma}(\widetilde{T}_n)$. 

While the $\widetilde{T}_n$ are random, the isomorphism type of each $\widetilde{T}_n$ is always the same because of ergodicity, and $\widetilde{T}_n$ can appear in finitely many $\Gamma$-equivalent ways in $T$. Let $T^1_n, T^2_n, \ldots T^{l(n)}_n$ denote the possible realizations of $\widetilde{T}_n$, and note that $\P[\widetilde{T}_n = T^i_n ]$ is the same for all $i \in \{1, \ldots , l(n)\}$.

Let $\varphi_n: H \to \Stab_{\Gamma}(\widetilde{T}_n)$ denote the restriction function:

\[\varphi_n (h) = h \vert_{\widetilde{T}_n}.\]

The function $\varphi_n$ is also random, but it only depends on $\widetilde{T}_n$, so once we condition $H$ on $\widetilde{T}_n$ the function $\varphi_n$ is well defined.

\begin{proposition} \label{proposition:action_on_T_i}
The random subgroup $\varphi_n \big( (H \mid \widetilde{T}_n = T^i_n) \big)$ is a fixed point free IRS in $\Stab_{\Gamma}(T^i_n)$.
\end{proposition}

\begin{proof}
For a fixed subgroup $L \leq \Gamma$ let $T_n(L)$ denote the deterministic subtree defined the same way as $\widetilde{T}_n$ was for $H$.
The set $\{L \leq \Gamma \mid T_n(L) = T^i_n\}$ is invariant under the conjugation action of $\Stab_{\Gamma}(T^i_n) \leq \Gamma$, so the invariance of the random subgroup $H$ implies the invariance of the conditioned subgroup $(H \mid \widetilde{T}_n = T^i_n)$. This IRS is fixed point free because all fixed points of $H$ are in $T_{\widetilde{C}}$. 
\end{proof}

\begin{remark}
One might be tempted to prove the more general Theorem \ref{theorem:finitary} by first proving the more transparent fixed point free case and then using Proposition \ref{proposition:action_on_T_i} on the individual subtrees, where $H$ acts fixed point freely. However, we do not see this approach to work. Instead with some mild additional technical difficulties we present the proof for the more general case. 
\end{remark}

\subsection{IRS's in finite subgroups of $\Gamma$}
Let $\Gamma_{n}$ stand for the subgroup consisting of elements of $\Gamma$ that only have nontrivial vertex permutations above $\mathcal{L}_n$.

\begin{lemma}\label{lemma:index_bound}
For $n$ large enough we have $[\Gamma_n : (K \cap \Gamma_n) ] \leq [\Gamma : K]$.
\end{lemma}

\begin{proof}
Fix a transversal for $K$. All elements in the transversal are finitary, so choose $n$ such that all are supported above $\mathcal{L}_n$. Then the translates of $(K \cap \Gamma_n)$ with this transversal cover $\Gamma_n$.
\end{proof}

\medskip

Let $\gamma \in \Gamma$, and $v \in \mathcal{L}_k$. The \emph{section} of $\gamma$ at $v$ is the automorphism $[\gamma]_v$ we get by restricting the portrait of $\gamma$ to the rooted subtree $T_v$ consisting of $v$ and its descendants. That is, the vertex permutations of $[\gamma]_v$ are defined by $(u)[\gamma]_v =(u)\gamma$ for every $u \in T_v$ and $(u)[\gamma]_v = \textrm{id}$ for every $u \notin T_v$. (Recall that $(w)\gamma$ denotes the vertex permutation of $\gamma$ at the vertex $w$, see Subsection \ref{subsec:aut_of_rooted_trees}.) We think of $[\gamma]_v$ as the automorphism on $T_v$ carried out by $\gamma$ before all the vertex permutations above the level $\mathcal{L}_k$ take place. 

Suppose $s \in \Gamma_{k}$, and let $L \subseteq \Gamma_{n}$ where $k < n$.  Let $\widetilde{L}$ denote the uniform random $\Gamma_{n}$-conjugate of $L$, which is an  IRS of $\Gamma_{n}$. Furthermore, assume that $\P[s \in \widetilde{L}] \geq c >0$, which is equivalent to 
\[\frac{\big|\{\gamma \in \Gamma_n \mid s^{\gamma} \in L \}\big|}{|\Gamma_n|} \geq c.\]

 Let $R \subseteq \Gamma_{n}$ be a transversal for the subgroup $\Rst_{\Gamma_n}(\mathcal{L}_k)$. By choosing the optimal one, we can find $\bar{\gamma} \in R$ such that   
\begin{equation}\label{eqn:many_conjugates}
\frac{\Big|\big\{ (\sigma_{v_1}, \ldots, \sigma_{v_{d^{k}}}) \in \Rst_{\Gamma_n}(\mathcal{L}_k) \ \big | \  s^{\bar{\gamma}(\sigma_{v_1}, \ldots, \sigma_{v_{d^{k}}})} \in L\big\}\Big|}{|\Rst_{\Gamma_n}(\mathcal{L}_k)|} \geq c.
\end{equation}

Here $(\sigma_{v_1}, \ldots, \sigma_{v_{d^{k}}})$ stands for the element of $\Rst_{\Gamma_n}(\mathcal{L}_k)$ that pointwise fixes $\mathcal{L}_{k}$, and has sections $\sigma_{v_i} \in \Rst_{\Gamma_n}(v_i)$ at the vertices $v_i \in \mathcal{L}_{k}$.

Let $\bar{s} = s^{\bar{\gamma}}$, and let the cycles of $\bar{s}$ on $\mathcal{L}_k$ be $C_1, \ldots C_r$, and let $C_i = (u^i_1 u^i_2 \dots u^i_{l(i)})$, $l(i)$ denotes the length of the cycle $C_i$, and $\bar{s}(u^i_j) = u^i_{j+1}$. We use the convention that $u^i_{l(i)+1} = u^i_1$. Assume that $l(1) \geq l(2) \geq \ldots \geq l(r)$ and let $t$ be the largest index for which $l(t) \geq 3$. Then $C = \supp(C_1) \cup \ldots \cup \supp(C_t) \subseteq \mathcal{L}_k$ is the union of $\bar{s}$-orbits of length at least 3 on $\mathcal{L}_k$. 

The next proposition shows that if $n$ is large enough, then $L$ has to contain the double commutator of some rigid level stabilizer under $C$, where the depth of this level does not depend on $n$.

\begin{proposition} \label{proposition:few_exceptions} Let $k, s$ and $c$ be fixed. Then there exists some $m > k$ and $n_0 > m$ such that for any $n \geq n_0$, $L$ and corresponding $\bar{\gamma}$ satisfying (\ref{eqn:many_conjugates}) above we have $\Rst''_{\Gamma_n}\big(\Sh_{\mathcal{L}_m}(C)\big) \subseteq L$.
\end{proposition}

\begin{proof}
Let  $\sigma = (\sigma_{v_1}, \ldots, \sigma_{v_{d^{k}}})$. Fix $\sigma_{v_i}$ for all $v_i \notin C$, and let the rest of the coordinates $\sigma_{u^i_j}$ vary over $\Rst_{\Gamma_n}(u^i_j)$. Choosing a maximum over all choices of the fixed $\sigma_{v_i}$ we can assume that 

\[\frac{\Big|\big\{ (\sigma_{u^i_j})_{i,j=1}^{t, l(i)} \in \Rst_{\Gamma_n}(C)\ \big | \  \bar{s}^{\sigma} \in L\big\}\Big|}{|\Rst_{\Gamma_n}(C)|} \geq c.\]  

Consider the conjugates $\bar{s}^{\sigma}$, more precisely what their sections are at the vertices $u^i_j$:

\begin{equation}\label{conjugatesection}
\left[\bar{s}^{(\sigma_{v_1}, \ldots, \sigma_{v_{d^{k}}})}\right]_{u^i_j} = \sigma_{u^i_j} \cdot [\bar{s}]_{u^i_j} \cdot (\sigma_{u^i_{j+1}})^{-1}.
\end{equation}

Fix one $\eta=(\eta_{v_1}, \ldots, \eta_{v_{d^{n-1}}}) \in \Rst_{\Gamma_n}(\mathcal{L}_k)$ with $\eta_{v_i} = \sigma_{v_i}$ for all $v_i \notin C$ and $\bar{s}^{\eta} \in L$. Let $\sigma_{u^i_j}$ run through $\Rst_{\Gamma_n}(u^i_j)$, and consider $\bar{s}^{\sigma} \cdot (\bar{s}^{\eta})^{-1}$. All these elements fix $\mathcal{L}_{k}$ pointwise, and their sections are

\[\big[ \bar{s}^{\sigma} \cdot (\bar{s}^{\eta})^{-1} \big]_{u^i_j} = \sigma_{u^i_j} \cdot [\bar{s}]_{u^i_j} \cdot (\sigma_{u^i_{j+1}})^{-1} \cdot \big(\eta_{u^i_j} \cdot [\bar{s}]_{u^i_j} \cdot (\eta_{u^i_{j+1}})^{-1}\big)^{-1}. \]
Observe that the sections are trivial over $v_i \notin C$. 

We will discard one vertex from each $C_i$, and focus on the sections we see on the rest. Let $D_{i} = C_i \setminus \{u^{i}_1\}$. 

Consider the sections of $\bar{s}^{\sigma}$ at the vertices in $D_{i}$ as the sections $(\sigma_{u^{i}_1}, \ldots, \sigma_{u^{i}_{l(i)}})$ run through $\Rst_{\Gamma_n}(C_i)$. We claim that the sections $\left( [\bar{s}^{\sigma}]_{u^{i}_2}, \ldots, [\bar{s}^{\sigma}]_{u^{i}_{l(i)}} \right)$ run through $\Rst_{\Gamma_n}(D_i)$. 

Indeed, given any sections $\left( [\bar{s}^{\sigma}]_{u^{i}_j} \right)_{j=2}^{l(i)}$, and any choice of $\sigma_{u^{i}_2}$ we can sequentially choose the $\sigma_{u^{i}_{j+1}}$ according to (\ref{conjugatesection}) to get the given sections at $j=2, 3, \ldots, l(i)$. The last choice is $\sigma_{u^{i}_1}$, which ensures $[\bar{s}^{\sigma}]_{u^{i}_{l(i)}}$ is correct. The last remaining section $[\bar{s}^{\sigma}]_{u^{i}_{1}}$ is already determined at this point, so we cannot hope to surject onto the whole $\Rst_{\Gamma_n}(C_i)$. 

We can do this independently for each $D_{i}$. Let \[D = \bigcup_{i} D_{i}.\] 

The sections of $\bar{s}^{\sigma}$ over the index set $D$ give $\Rst_{\Gamma_n}(D)$ as the sections $(\sigma_{u^{i}_j})$ run through $\Rst_{\Gamma_n}(C)$. 

 Whenever $\bar{s}^{\sigma} \in L$ we have $\bar{s}^{\sigma} \cdot (\bar{s}^{\eta})^{-1} \in L_0$, where $L_0 \subseteq \Stab_{L}(\mathcal{L}_{k})$ is the set of elements with trivial sections outside $C$. The fact that a fixed positive proportion of the $\bar{s}^{\sigma}$ are in $L$ ensures that a fixed proportion of the elements of $\Rst_{\Gamma_n}(D)$ are seen in $L_0$. Let $\pi_D: \Stab_L(\mathcal{L}_k) \to \Gamma_{n-k}^{D}$ denote the projection to the coordinates in $D$. Formally we get

\[\big|\pi_D (L_0)\big| \geq c \cdot |\Rst_{\Gamma_n}(D)|.\]

As $\Gamma$ is a regular branch group over some $K$, we have $\pi_D (L_0) \leq (\Gamma_{n-k})^{|D|}$, and $(K \cap \Gamma_{n-k})^{|D|} \leq \Rst_{\Gamma_n}(D)$. Using Lemma \ref{lemma:index_bound} we can bound the index of $\pi_D (L_0)$ in $\Gamma_{n-k}^{|D|}$:

\[\big[ (\Gamma_{n-k})^{|D|} : \pi_D (L_0) \big] = \frac{\left|(\Gamma_{n-k})^{|D|}\right|}{|\pi_D (L_0)|} \leq \left\lceil\frac{1}{c}\right\rceil \cdot \frac{\left|(\Gamma_{n-k})^{|D|}\right|}{|\Rst_{\Gamma_n}(D)|} \leq \]
\[ \leq \left\lceil\frac{1}{c}\right\rceil \cdot \frac{\left|(\Gamma_{n-k})^{|D|}\right|}{\left|(K \cap \Gamma_{n-k})^{|D|}\right|} \leq \left \lceil\frac{1}{c}\right\rceil \cdot [\Gamma: K]^{|D|}.\] 

In any group a subgroup of finite index $M$ contains a normal subgroup of index at most $(M!)$. That is, we can find some $N \vartriangleleft (\Gamma_{n-k})^{|D|}$ such that $ N \leq \pi_D (L_0)$ and \[\big[(\Gamma_{n-k})^{|D|} : N \big] \leq \left( \left \lceil\frac{1}{c}\right\rceil \cdot [\Gamma: K]^{|D|}\right)!\]

The bound on the index of $N$ does not depend on $n$, only on $k$, $s$ and $c$. The bounded index ensures, that we can find some $m>k$ and $n_0>m$ with the following properties. For $n \geq n_0$ and for each index $u \in D$ we can find an element $\varphi \in N$ such that $\pi_{u}(\varphi) \notin \Stab_{\Gamma_{n-k}}\big(\mathcal{L}_{m-k}(T_u)\big)$. Here $\pi_u$ denotes the projection to the $u$ coordinate in $\Gamma_{n-k}^{D}$. Moreover, we choose $n_0$ large enough so that $\Gamma_{n_0-k}$ acts transitively on $\mathcal{L}_{m-k}(T_u)$.

Using Grigorchuk's standard argument from \cite[Lemma 5.3]{bartholdi2003branch} and \cite[Theorem 4]{grigorchuk2000just} we pick some $w \in \mathcal{L}_{m-k}(T_u)$ not fixed by $\varphi$, elements $f$ and $g$ from $\Rst_{\Gamma_n}(uw)$ and argue that the commutator $[[\varphi, f], g] = [f,g]$ is in $N$. This shows $\Rst'_{\Gamma_n}(uw) \subseteq N$. If $n \geq n_0$ then $\Gamma_{n-k}$ is transitive on $\mathcal{L}_{m-k}(T_u)$, so we get \[\Rst'_{\Gamma_n}\big(\mathcal{L}_{m-k}(T_u)\big) \subseteq N.\]

Repeating the argument of the previous paragraph for all $u \in D$ we get 

\[\Rst'_{\Gamma_n} \big(\Sh_{\mathcal{L}_m}(D)\big) \subseteq N \subseteq \pi_D (L_0).\]

We now repeat this discussion, but we discard different points from the orbits: let $E_{i} = (C_{i}) \setminus \{u^{i}_2\}$ and $E = \bigcup_{i} E_{i}$.  We have 

\[\Rst'_{\Gamma_n} \big(\Sh_{\mathcal{L}_m}(E)\big) \subseteq \pi_E (L_0).\]

We claim that $\Rst''_{\Gamma_n} \big(\Sh_{\mathcal{L}_m}(D \cap E)\big) \subseteq  L$. Indeed, let $u^i_j \in C_i$, $j \neq 1,2$. By the above we see that for any $\varphi \in \Rst'_{\Gamma_n}\big(\Sh_{\mathcal{L}_m}(u^i_j)\big)$ we have $h_1 \in L_0$ such that $\pi_D(h_1)_{u^{i}_j}=\varphi$ and all other coordinates of $\pi_D(h_1)$ are the identity. Similarly we have $h_2 \in L_0$ such that $\pi_E(h_2)_{u^{i}_j}=\psi$ and all other coordinates of $\pi_E(h_1)$ are the identity. Since $\mathcal{L}_{k} \setminus  D$ and $\mathcal{L}_{k} \setminus E$ are disjoint the commutator $[h_1, h_2] \in L_0$ has all identity coordinates except for the one corresponding to $u^{i}_j$ which is $[\varphi, \psi]$. 

We have managed to take care of the points $u^{i}_j$ where $j \neq 1,2$. To cover the remaining points as well we need one more way to discard points from the orbits. Namely $F$, where we discard the third vertex $u^{i}_3$ from every $C_i$. Using the fact that $(D \cap E) \cup (E \cap F) \cup (D \cap F) = C$ we get that $\Rst''_{\Gamma_n}\big(\Sh_{\mathcal{L}_m}(C)\big) \subseteq L$, which finishes the proof. 
\end{proof}

\subsection{Proof of the main result}
\label{subsection:proof_of_finitary}

\begin{proof}[Proof of Theorem \ref{theorem:finitary}]
During the proof we will have to choose deeper and deeper levels in $T$. For the convenience of the reader we summarized these choices in Figure \ref{fig:struc}. 

Let $\Gamma_n \subseteq \Gamma$ denote the subgroup of elements of $\Gamma$ that have trivial vertex permutations below the $n^{\mathrm{th}}$ level. Suppose that $H$ is an ergodic IRS of $\Gamma$.

By Theorem \ref{theorem:clopen} we know that the all nontrivial orbit-closures of $H$ on $\partial T$ are clopen. For every clopen set $C$ there exists a smallest integer $k_C$ such that $C$ is the union of shadows of points on $\mathcal{L}_{k_C}$. Clearly $k_C$ does not change when $C$ is translated by some automorphism. For the random subgroup $H$ and a fixed $k_0 \in \N$ we can collect the clopen sets $C$ from $\mathcal{O}_H$ with $k_C < k_0$, let $C_{H, k_0}$ be the union of these. This set moves together with $H$ when conjugating by some $\gamma \in \Gamma$:

\[C_{H^{\gamma}, k_0} = (C_{H, k_0})^{\gamma}.\]

For $n \geq k_0$ let $V_n \subset \mathcal{L}_{n}$ be the set of points whose shadow make up $C_{H, k_0}$. As $C_{H, k_0}$ moves with $H$, so does $V_n$. 
$V_{k_0}$ is a union of orbits of $H$, let those orbits be denoted $V_{k_0}^i$, where $i\in \{1, \ldots, j\}$ and 

\[V_{k_0} = \bigcup_{i=1}^j V_{k_0}^i.\]

Let $V_n^i = \Sh_{\mathcal{L}_n}(V_{k_0}^i)$. The fact that $H$ acts minimally on the components of $C_{H, k_0}$ translates to saying that $H$ acts transitively on each $V_n^i$. Notice that since we collected clopen sets $C$ with $k_C$ strictly less then $k_0$ we ensured that $V_{k_0}^i$ contains at least $d$ points for all $i$. 

For every realization of $H$ we can choose finitely many elements of $H$ that already show that $H$ acts transitively on $V_{k_0}^i$ for all $i$. These finitely many elements are all finitary, so there is some $n_H$, which might depend on the realization of $H$,  such that all those finitely many elements are in $\Gamma_{n_H}$. 

This function $n_H$ is not necessarily conjugation-invariant, so it need not be constant merely by ergodicity. However, for any $\varepsilon > 0$, one can find some $k \geq k_0$ such that the $V_{k_0}^i$ are distinct orbits of $H_{k} = H \cap \Gamma_{k}$ on $\mathcal{L}_{k_0}$ with probability at least $1-\varepsilon$. This $k$ is a deterministic number, it does not depend on the realization of $H$, only on the choice of $\varepsilon$.

Enlist all the possible subsets $S_1, \ldots, S_N$ of $\Gamma_{k}$ that generate a realization of the $V_{k_0}^i$ as orbits on $\mathcal{L}_{k_0}$. Clearly there are finitely many. The probability that $S_i \subseteq H$ cannot be 0 for all $i$, otherwise we would contradict the previous paragraph. So we can find some finite set $S$ of elements of $\Gamma_{k}$ and some sets $U_{k_0}^i \subseteq \mathcal{L}_{k_0}$ such that the $U_{k_0}^i$ are a realization of the $V_{k_0}^i$, $S$ is in $H$ with probability $p > 0$ and the $U_{k_0}^i$ are orbits of $S$.

By replacing $S$ with $\langle S \rangle$ we may assume that $S$ is a subgroup of $\Gamma_{k}$, as $S \subseteq H$ and $\langle S \rangle \subseteq H$ are the same events. 

\begin{figure}[h]
   \centering
   \includegraphics[width=14.5cm]{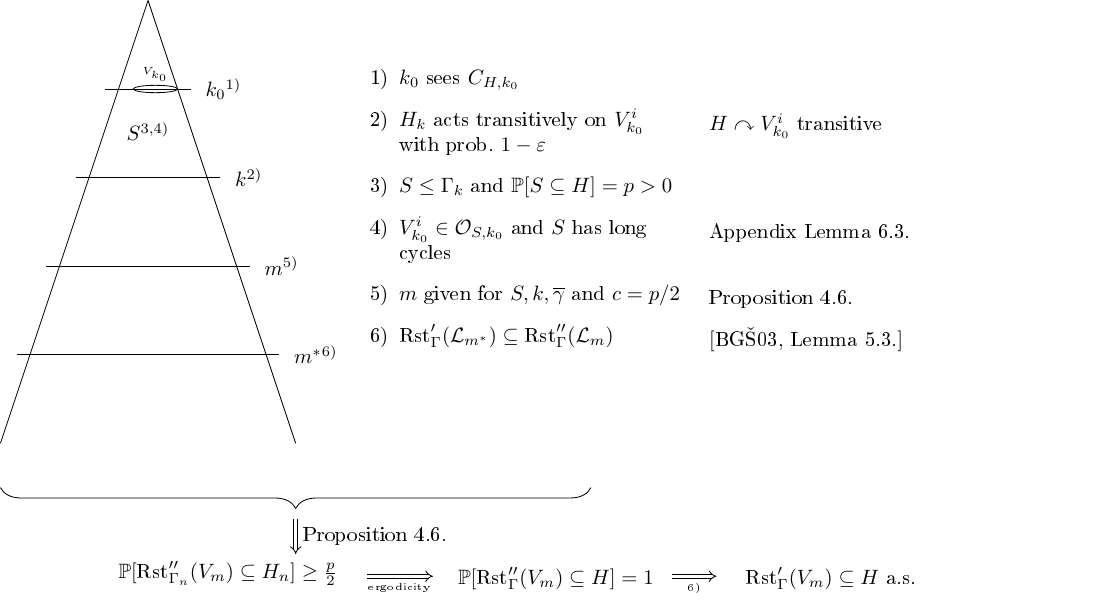}
   \caption{Choice of levels}
   \label{fig:struc}
\end{figure}

As  $|U_{k_0}^i| \geq d$, we know that all vertices of $U_{k_0}=\bigcup_{i=1}^j U_{k_0}^i$ are moved by some $s \in S$. As a consequence the same holds for $U_k$: for every vertex $v \in U_{k}$ there is some $s \in S$ such that $v \neq v^s$. However, we will need a stronger technical assumption on $S$ to make our argument work. We will assume that for every $v \in U_{k}$ we can find some $s \in S$ such that $v$, $v^s$ and $v^{s^2}$ are distinct, that is $v$ is part of a cycle of length at least 3 in the cycle decomposition of $s$. In Lemma \ref{lemma:technical_assumption} in the Appendix we show that one can indeed find such a $k$ and $S$.

Let $H_n = H \cap \Gamma_n$, for $n \geq k$. The random subgroup $H_n$ is clearly an IRS of $\Gamma_n$, however it need not be ergodic, i.e.\ the uniform random conjugate of a fixed subgroup in $\Gamma_n$. As $S \leq \Gamma_{k} \leq \Gamma_n$ we have $\P[S \leq H_n]=p$. 

\begin{lemma} \label{lemma:nice_ergodic_components}
In the ergodic decomposition of $H_{n}$ the measure of components that contain $S$ with probability at least $p/2$ is at least $p/2$. 
\end{lemma}

\begin{proof}
Denote the ergodic components of $H_{n}$ by $H^{1}_{n}, \ldots, H^{r}_{n}$. Assume $H^{i}_{n}$ has weight $q_i$ in the decomposition, and contains $S$ with probability $p_i$. By ordering appropriately we can also assume $p_1, \ldots, p_l \geq p/2$ and $p_{l+1}, \ldots, p_{r} < p/2$. 

\[p = \sum_{i=1}^{r} q_i p_i \leq \left( \sum_{i=1}^{l} q_i \right) \cdot 1 + \left( \sum_{i=l+1}^{r} q_i \right) \cdot \frac{p}{2} \leq \left( \sum_{i=1}^{l} q_i \right) + \frac{p}{2},\]

\[\frac{p}{2} \leq \sum_{i=1}^{l} q_i.\]

So the weight of components containing $S$ with probability at least $p/2$ is at least $p/2$.  
\end{proof}

Choose an ergodic component of $H_{n}$ which contains $S$ with probability at least $p/2$. This ergodic component is the uniform random conjugate of a fixed subgroup $L \leq \Gamma_{n}$. 

We have $\P[S \in \widetilde{L}] \geq \frac{p}{2} >0$. In other words $L$ contains at least a $p/2$ proportion of the $\Gamma_n$-conjugates of $S$. By a ''maximum is at least as large as the average'' argument we can find some $\bar{\gamma}$ from the transversal of $\Rst_{\Gamma_n}(\mathcal{L}_k)$ such that   
\[\frac{\bigg|\big\{ (\sigma_{v_1}, \ldots, \sigma_{v_{d^{k}}}) \in \Rst_{\Gamma_n}(\mathcal{L}_k) \ \big | \  S^{\bar{\gamma}(\sigma_{v_1}, \ldots, \sigma_{v_{d^{k}}})} \in L\big\}\bigg|}{|\Rst_{\Gamma_n}(\mathcal{L}_k)|} \geq \frac{p}{2}.\]

We now use Proposition \ref{proposition:few_exceptions} for all $s \in S$ with $k$, $\bar{\gamma}$ defined above and $c = \frac{p}{2}$. As the cycles of length at least 3 of elements of $S^{\bar{\gamma}}$ cover $(U_{k})^{\bar{\gamma}}$ we get that for some fixed $m$ and large enough $n$ we have
\[
\Rst''_{\Gamma_n} \big((U_{m})^{\bar{\gamma}}\big)   \subseteq L.
\]  
It is clear that $(U_{k_0})^{\gamma}$ is the realization of $V_{k_0}$ corresponding to the realization $L$ of $H_n$, so we (almost surely) have $\Rst''_{\Gamma_n} (V_{m})  \subseteq \widetilde{L}$. The last inclusion holds for any ergodic component $\widetilde{L}$ oh $H_n$ containing $S$ with probability at least $p/2$. By Lemma \ref{lemma:nice_ergodic_components} the measure of these components of $H_n$ is at least $p/2$. Therefore: \[\P\big[\Rst''_{\Gamma_n} (V_m)   \subseteq H_n\big] \geq \frac{p}{2}.\]
 
As $\big(\Rst''_{\Gamma_n} (V_m)   \subseteq H_n\big) \Leftrightarrow \big( \Rst''_{\Gamma_n} (V_m)   \subseteq H \big)$ we have 

\[ \P\big[ \Rst''_{\Gamma_n} (V_m) \subseteq H\big]   \geq  \frac{p}{2}.\] 

We get this for all $n$ large enough. Since $\Rst''_{\Gamma_n} (V_m) \subseteq \Rst''_{\Gamma_{n+1}} (V_m)$ the events in question form a decreasing chain, and for the intersection we get

\[ \P\big[ \Rst''_{\Gamma} (V_m)  \subseteq H\big] \geq  \frac{p}{2}.\] 

As $H$ is ergodic the above implies \[\P\big[ \Rst''_{\Gamma} (V_m)  \subseteq H\big]= 1.\] 
Clearly $\Rst''_{\Gamma} (\mathcal{L}_m) \vartriangleleft \Gamma$, so using \cite[Lemma 5.3]{bartholdi2003branch} we can find some $m^* \geq m$ such that $\Rst'_{\Gamma} (\mathcal{L}_{m^*}) \subseteq \Rst''_{\Gamma}(\mathcal{L}_m)$. This also means that $\Rst'_{\Gamma} (V_{m^*}) \subseteq \Rst''_{\Gamma} (V_m)$, so 

 \[\P\big[ \Rst'_{\Gamma} (V_{m^*})  \subseteq H\big]= 1.\]

The number $m^*$ only depended on the IRS $H$ and the choice of $k_0$. Repeating this argument for all $k_0 \in \N$ covers all clopen sets from $\mathcal{O}_H$, which as discussed in part \ref{subsection:decomposition} proves Theorem \ref{theorem:finitary}.
\end{proof}

\section{Corollaries of Theorem \ref{theorem:finitary}} \label{section:corollaries}

In this section we prove Theorem \ref{theorem:atomic} and sketch the proof of Theorem \ref{theorem:branch_group_closure}.

\subsection{Fixed point free IRS's}

To motivate the following result let us recall Theorem \ref{theorem:alternating_finitary_ffp}, which states that any fixed point free, ergodic IRS of $\Alt_f (T)$ with $d \geq 5$ contains a whole level stabilizer, in particular $H$ is a random conjugate of a finite indexed subgroup. In other words the measure defining the IRS is atomic. As it turns out the fixed point free case of Theorem \ref{theorem:finitary} implies this for fixed point free ergodic IRS's of countable, finitary regular branch groups as well.

\medskip

\begin{proof}[Proof of Theorem \ref{theorem:atomic}]
By Theorem \ref{theorem:finitary} we know that an ergodic almost surely fixed point free IRS $H$ of a finitary regular branch group $\Gamma$ contains $\Rst'_{\Gamma}(\mathcal{L}_m)$ for some $m \in \N$. 

IRS's of $\Gamma$ containing the normal subgroup $\Rst'_{\Gamma}(\mathcal{L}_m)$ are in one-to-one correspondence with IRS's of the quotient $G=\Gamma/\Rst'_{\Gamma}(\mathcal{L}_m)$, which in this case is of the form $A \rtimes F$ where $A$ is the abelian group $\Rst_{\Gamma}(\mathcal{L}_m) / \Rst'_{\Gamma}(\mathcal{L}_m)$, and $F$ is the finite group $\Gamma / \Rst_{\Gamma}(\mathcal{L}_m)$. As $\Gamma$ is assumed to be finitary both $\Gamma$ and $G$ are countable.

Let $\widehat{H} = H / \Rst'_{\Gamma}(\mathcal{L}_m) \leq G$ be the image of $H$ in $G$. It is an ergodic IRS of $G$. Let $\widehat{H}_0 = \widehat{H} \cap A$, which is also an ergodic IRS of $G$. We see that $\widehat{H}_0 \subseteq A$ is an ergodic random subgroup with distribution invariant under conjugation by elements of $G$. As $A$ is abelian and $F$ is finite, it is clearly the uniform random $F$-conjugate of some subgroup $L_0 \leq A$. This shows that $\widehat{H}_0$ can only obtain finitely many possible values.

We claim that once $\widehat{H}_0$ is fixed, there are only countably many possible choices for $\widehat{H}$. Indeed we have to choose a coset of $\widehat{H}_0$ in $G$ for all $f \in F$, which can do in only countably many different ways. 

This shows that the support of $\widehat{H}$ is countable, but there is no ergodic invariant measure on a countably infinite set, so the support is finite. 
\end{proof}

\subsection{IRS's in non-finitary branch groups}

In this subsection we will sketch the proof of Theorem \ref{theorem:branch_group_closure}. This theorem is not a direct consequence (as far as we see) of Theorem \ref{theorem:finitary}, but one can alter the proof to obtain the desired theorem. 
First of all let us fix $\pi_n:\Gamma\to S_d^{\mathrm{wr}(n)}$ to be the projection from $\Gamma$ to the automorphism group of the $d$-ary tree of depth $n$, which is the restriction of elements to the first $n$ levels.
The main conceptional difference is that we are trying to understand the group $\Gamma$ through the groups $\pi_n(\Gamma)$ instead of $\Gamma_n$. The statement that we conclude in this case is weaker.

Our aim is to present only the spine of the proof, as the reasoning is very similar to the proof of Theorem \ref{theorem:finitary} and we leave the details to the reader. In fact some technical details such as the ergodicity of $H_n$ and the fact that $H_n$ already acts transitively on the $V^i_n$ makes this proof easier. 

\medskip
\begin{proof}[Proof of Theorem \ref{theorem:branch_group_closure}]
Let $G_n=\pi_n(\Gamma)$, fix $k \in\mathbb{N}$ and let $C_{H,k}$ be the union of clopen orbit-closures $C$ from $\mathcal{O}_H$ in $\partial T$ with $k_C<k$. For any $n\ge k$ let $V_n\subseteq \mathcal{L}_n$ be the set of points whose shadow make up $C_{H,k}$. We can decompose $V_{k}$ into $H$-orbits, denoted by 
\[
  V_{k}=\cup_{i=1}^jV_{k}^i.
\]

Observe that for any realization of $H$ one can find at most $|V_{k}|$ many elements in $H$ that already show that $H$ acts transitively on each $V_{k}^i$. This means that we can find an $S\subset \Gamma$ of size at most $|V_{k}|$, such that $S$ \emph{induce} a realization of the $V_{k}^i$ as orbits on $\mathcal{L}_{k}$,   
\[
  \mathbb{P}[S\subseteq H]=p>0.
\]

The orbits \emph{induced} by $S$ on $\mathcal{L}_{k}$ are the orbits of $\langle S \rangle$. Though $\langle S \rangle$ might be infinite, we only have to look for our $S$ among sets of size at most $|V_k|$, which ensures that one of those is contained in $H$ with positive probability. (Of course $S \subseteq H$ and $\langle S \rangle  \subseteq H$ are equivalent, so in fact we see that $\mathbb{P}[\langle S \rangle \subseteq H]=p>0$, but we will only need $S$ in this proof.)

Denote by $U_{k}^i$ the realization of $V_{k}^i$ that $S$ induces as orbits. As before we can ensure that  for any $v\in U_{k}$ there is an $s\in S$, such that $v$, $v^s$ and $v^{s^2}$ are distinct by replacing $k$ and $S$ if necessary. (See the Remark after the proof of Lemma \ref{lemma:technical_assumption} in the Appendix.)

For every $n$ let  $H_n=\pi_n(H) \le G_n$. The random subgroup $H_n$ is an ergodic IRS of $G_n$, therefore there exists an $L_n\le G_n$ such that $H_n$ is an uniform random conjugate of $L_n$. Since

\[
  \mathbb{P}[\pi_n(S)\subseteq \widetilde{L_n}]=\mathbb{P}[\pi_n(S)\subseteq H_n]\ge p,
\]
we have an element $\overline{\gamma}$ from the transversal of $\Rst_{G_n}(\mathcal{L}_{k})$ in $G_n$ such that
\[
  \frac{\bigg|\big\{ (\sigma_{v_1}, \ldots, \sigma_{v_{d^{k}}}) \in \Rst_{G_n}(\mathcal{L}_k) \ \big | \  \pi_n(S)^{\bar{\gamma}(\sigma_{v_1}, \ldots, \sigma_{v_{d^{k}}})} \in L_n\big\}\bigg|}{|\Rst_{G_n}(\mathcal{L}_k)|} \geq p.
\]

By following the argument in Proposition \ref{proposition:few_exceptions} but replacing $\Gamma_n$ by $G_n$ one can prove that there exists some $m$ such that for any $n$ large enough
\[
  \Rst_{G_n}''((U_m)^{\bar{\gamma}})\subseteq L_n.
\]
Therefore
\[
  \mathbb{P}[\Rst''_{G_n}(V_m)\subseteq H_n]\ge p>0,
\]
which by ergodicity implies
\[
  \mathbb{P}[\Rst''_{G_n}(V_m)\subseteq H_n]=1.
\]
Again we can find an $m^*\ge m$, such that $\Rst'_\Gamma(V_{m^*})\subseteq \Rst''_\Gamma(V_{m})$, therefore
\[
  \mathbb{P}[\pi_n(\Rst'_{\Gamma}(V_{m^*}))\subseteq \pi_n(H)]=1.
\]

This means that for any $g\in \Rst'_{\Gamma}(V_{m^*})$ there exists a sequence $h_n\in H$, such that $\pi_n(h_n)=\pi_n(g)$,  which implies that $\overline{\Rst'_{\Gamma}(V_{m^*})}\subseteq \overline{H}$ with probability 1.

On the other hand $\overline{\Rst'_{\Gamma}(V_{m^*})}\supseteq \overline{\Rst_{\Gamma}(V_{m^*})}'$. We claim that  $\Rst_{\overline{\Gamma}}(V_{m^*})=\overline{\Rst_{\Gamma}(V_{m^*})}$. Indeed, $\Rst_{\overline{\Gamma}}(\mathcal{L}_{m^*})$ is finite index in $\overline{\Gamma}$ which implies that it is open. Using this one can show that $\Rst_{\overline{\Gamma}}(\mathcal{L}_{m^*}) = \overline{\Rst_{\Gamma}(\mathcal{L}_{m^*})}$, which implies the same for $V_{m^*} \subseteq \mathcal{L}_{m^*}$. 

Putting this together we get
\[
  \Rst'_{\overline{\Gamma}}(V_{m^*})\subseteq \overline{H}
\]
with probability 1.
\end{proof}
 
\medskip

Note that this result on closures is possibly weaker than our earlier results. It is not clear even in the fixed point free case in $\Alt_f(T)$ if for some $L \leq \Alt_f(T)$ the closure $\overline{L}$ containing a level stabilizer implies the same for $L$.

\begin{problem} \label{problem:surjective_subgroup_whole}
Let $L \leq \Alt_f(T)$ be a subgroup such that $\pi_n(L) = A_d^{\mathrm{wr}(n)}$ for all $n$. Does it follow that $L=\Alt_f(T)$?
\end{problem}

In other words: is there a subgroup  $L \neq \Alt_f(T)$ which is dense in $\Alt(T)$? We saw that this cannot happen with positive probability when $L$ is invariant random.

The answer to Problem \ref{problem:surjective_subgroup_whole} is negative in the case of $\Aut(T)$. In the case of the binary tree let $L$ be the subgroup of elements with an even number of nontrivial vertex permutations. Generally for arbitrary $d$ let $L$ be the subgroup of elements whose vertex permutations multiply up to an alternating element. This $L$ is not the whole group, yet dense in $\Aut(T)$. Of course the really relevant question in this case would involve the containment of derived subgroups of level stabilizers.

\section{Appendix} \label{section:appendix}

In this section we prove the technical statements that we postponed during the rest of the paper. 

\subsection{Measurability of maps}

 \begin{proof}[Proof of Lemma \ref{lemma:measurable_fixedpoint_set}]

A closed subset $C$ can be approximated on the finite levels. Define $C_n$ to be the set of vertices $v$ on $\mathcal{L}_n$ with $\Sh(v) \cap C \neq \emptyset$. The sets $C_n$ correspond to the $(1/2^n)$-neighborhoods of $C$ in $\partial T$. 

We show that any preimage of a ball in $(\mathcal{C}, d_H)$ is measurable in $\Sub_{\Gamma}$. Let $C \in \mathcal{C}$ and $n\in \mathbb{N}$ be fixed. Then the ball
 \[
  B_{1/{2^n}}(C)=\{C' \in\mathcal{C}~|~ C_n=C'_n\}, 
 \]
 therefore its preimage is
 \[
  X=\{H\in \Sub_\Gamma ~|~ \Fix(H)_n=C_n\}.
 \]

We say that a finite subset $S \subseteq \Gamma$ witnesses $C_n$, if the subgroup they generate has no fixed points in $\mathcal{L}_n \setminus C_n$. If $C_n$ has no witness at all (this happens if there is a $\Gamma$-fixed vertex in $\mathcal{L}_n \setminus C_n$), then $X$ is empty, because for all $H \in \Sub_\Gamma$ we have $\Fix(H)_n \cap (\mathcal{L}_n \setminus C_n)\neq \emptyset$. If $C_n$ has a witness, it clearly has a witness of cardinality at most $|\mathcal{L}_n \setminus C_n|$. Let $W_{C_n}$ be the set of possible witnesses of $C_n$ of size at most $|\mathcal{L}_n \setminus C_n|$:
 \[
  W_{C_n}=\big\{S\subseteq \Gamma ~\big|~ |S|\le |\mathcal{L}_n \setminus C_n|\textrm{ and $S$ has no fixed points in }\mathcal{L}_n \setminus C_n \big\}.
 \]

Let us define $F_{C_n} \subseteq \Gamma$ to be the set of forbidden group elements, which do not fix $C_n$ pointwise. These are the elements that cannot be in any $H \in X$.
 
 Observe that both $W_{C_n}$ and $F_{C_n}$ are countable, since $\Gamma$ is countable, and $X$ can be obtained as
 
 \[
  X=\bigcup_{S\in W_{C_n}}\bigcap_{g \in F_{C_n}}\{H\in \Sub_\Gamma ~|~ S\subseteq H, g\notin H\}.
 \]

The sets $\{H\in \Sub_\Gamma ~|~ S\subseteq H, g\notin H\}$ are cylinder sets in the topology of $\Sub_{\Gamma}$, so the above expression shows that $X$ is measurable.
\end{proof}

\bigskip

\begin{proof}[Proof of Lemma \ref{lemma:measurable_orbit_closures}]
 
 To prove that the map is measurable, it is enough to show that any preimage of a ball is measurable in $\Sub_{\Gamma}$.
 So let $P\in\mathcal{O}$ and $n\in \mathbb{N}$ be fixed. Then the ball
 \[
  B_{1/{2^n}}(P)=\{Q\in\mathcal{O}~|~Q_n=P_n\}, 
 \]
 therefore its preimage is
 \[
  X=\{H\in \Sub_\Gamma ~|~ \left(\mathcal{O}_H\right)_n=P_n\}.
 \]
 
We say that a finite subset $S \subseteq \Gamma$ witnesses $P_n$, if the subgroup they generate induces the same orbits on $\mathcal{L}_n$, that is $\mathcal{O}_{\langle S\rangle,\mathcal{L}_n}=P_n$. Clearly every $P_n$ has a witness of cardinality at most $|\mathcal{L}_n|$. Let $W_{P_n}$ be the set of possible witnesses of $P_n$ of size at most $|\mathcal{L}_n|$:
 \[
  W_{P_n}=\big\{S\subseteq \Gamma ~\big|~ |S|\le |\mathcal{L}_n|\textrm{ and } \mathcal{O}_{\langle S\rangle,\mathcal{L}_n}=P_n\big\}.
 \]

Let us define $F_{P_n} \subseteq \Gamma$ to be the set of forbidden group elements, which do not preserve $P_n$. In other words these are the elements that cannot be in any $H \in X$.
 
 Observe that both $W_{P_n}$ and $F_{P_n}$ are countable, since $\Gamma$ is countable, and $X$ can be obtained as
 
 \[
  X=\bigcup_{S\in W_{P_n}}\bigcap_{g \in F_{P_n}}\{H\in \Sub_\Gamma ~|~ S\subseteq H, g\notin H\}.
 \]

The sets $\{H\in \Sub_\Gamma ~|~ S\subseteq H, g\notin H\}$ are cylinder sets in the topology of $\Sub_{\Gamma}$, so the above expression shows that $X$ is measurable.
\end{proof}

\subsection{Technical assumption in Theorem \ref{theorem:finitary}} \label{appendix:technical}

First we prove a lemma on intersection probabilities.

\begin{lemma} \label{lemma:intersection_probability}
Let $B_1, \dots B_r$ be measurable subsets of the standard probability space $(X, \mu)$ with $\mu(B_j)=p$ for all $j$, and $r=\left \lceil \frac{2}{p} \right \rceil$. Then there is some pair $(j,l)$ such that $\mu (B_j \cap B_l) \geq \frac{p^3}{6}$.
\end{lemma}

\begin{proof}
Let $\chi_{B}$ denote the characteristic function of the measurable set $B$. Let $D_{l}$ denote the set of points in $X$ that are covered by at least $l$ sets from $B_1, \dots B_r$. Then

\[\sum_{j=1}^{r} \chi_{B_j} = \sum_{l=1}^{r} \chi_{D_l},\]

\[ \int_{X}\sum_{j=1}^{r} \chi_{B_j} \ d \mu = \sum_{j=1}^{r} \mu(B_j) = rp,\]

\[ rp = \int_{X} \sum_{l=1}^{r} \chi_{D_l} \ d \mu = \sum_{l=1}^{r} \mu(D_l).\]
We have $D_1 \supseteq D_2 \ldots \supseteq D_r$, so $1 \geq \mu (D_1) \geq \mu(D_2) \ldots \geq \mu(D_r)$. 

\[ rp = \sum_{l=1}^{r} \mu(D_l) \leq 1 + (r-1)\mu(D_2).\]

\[\mu(D_2) \geq \frac{rp-1}{r-1}.\]
The set $D_2$ is covered by the $B_j \cap B_l$, so 

\[ \max_{j,l} \mu(B_j \cap B_l) \geq \frac{\mu(D_2)}{{r \choose 2}} \geq \frac{rp-1}{{r \choose 2} (r-1)} \geq \frac{1}{\left(\frac{\left(\frac{2}{p} + 1\right) \left(\frac{2}{p} \right)}{2}\right)\left(\frac{2}{p}\right)} \geq \frac{p^3}{2(p+2)} \geq \frac{p^3}{6}.\]
\end{proof}

\bigskip

We also prove that one can find  a lot of elements of order at least 3 in weakly branch groups.

\begin{lemma}\label{lemma:high_order_element}
 Let $G$ be a weakly branch group. Then for any $v\in T$ there is a $g\in\Rst_G(v)$ of order at least 3.
\end{lemma}
\begin{proof}
As $G$ is weakly branch, $\Rst_G(v)$ is not the trivial group. By contradiction let us assume that any nontrivial element of $\Rst_G(v)$ has order 2. Let a nontrivial element $g \in \Rst_G(v)$ be fixed. We can find descendants $u_1\neq u_2$ of $v$, such that $u_1^g=u_2$ . Let us choose a nontrivial element $h\in\Rst_G(u_1)$. As $h\in\Rst_G(v)$,  by our assumption it has order 2.

 We claim that $hg\in\Rst_G(v)$ has order at least three. To prove this we will find a vertex which has an orbit of size at least 3. Let $w_1\neq w_2$ be descendants of $u_1$, such that $w_1^h=w_2$. 
 Since $g$ maps descendants of $u_1$ to descendants of $u_2$, we have $w_1^{hg}=w_2^g=t_2\neq w_1,w_2$. Then 
 $w_1^{(hg)^2}=t_2^{hg}=t_2^g=w_2\neq w_1$. We see that $w_1, w_1^{hg}$ and $w_1^{(hg)^2}$ are pairwise distinct, therefore the order of $hg$ is at least 3.
\end{proof}

\bigskip

Now we will prove that the technical assumption we assumed in the proof of Theorem \ref{theorem:finitary} can be satisfied. We remind the reader that in the setting of Theorem \ref{theorem:finitary} the following were established:

\begin{enumerate} [(1)]
\item The random sets $(V^{k_0}_{1}, \ldots V^{k_0}_{j})$ are orbits of $H$ on $\mathcal{L}_{k_0}$;
\item $\Sh(V^{k_0}_1), \ldots, \Sh(V^{k_0}_j)$ are orbit-closures of $H$ on $\partial T$ and their union is $C_{H,k_0}$ almost surely;
\item $S \leq \Gamma_{k}$ is a finite subgroup with $p=\P[S \subseteq H]$ positive;
\item $(U^{k_0}_{1}, \ldots U^{k_0}_{j})$ are a realization of $(V^{k_0}_{1}, \ldots V^{k_0}_{j})$, and $S$ acts transitively on the $U^{k_0}_{i}$.
\item $V^{n}_{i}=\Sh_{\mathcal{L}_n} (V^{k_0}_{i})$ and $U^{n}_{i}=\Sh_{\mathcal{L}_n} (U^{k_0}_{i})$.
\end{enumerate}

\begin{lemma} \label{lemma:technical_assumption}
By possibly replacing $k$, $S$ and $p$ we can assume that for every $u \in U^{k}_{i}$ we can find some $s \in S$ such that $u$, $u^s$ and $u^{s^2}$ are distinct.
\end{lemma}

\begin{remark}
In the case when $d$ is not a power of 2 it can be shown that Lemma \ref{lemma:technical_assumption}
 is implied by the earlier properties, simply because a transitive permutation group with all nontrivial elements being fixed point free and of order 2 can only exist on $2^k$ points. For the case when $d$ is a power of 2 however we can only show Lemma \ref{lemma:technical_assumption} by a probabilistic argument and by increasing $k$ and $S$ if necessary.
\end{remark}

\begin{proof}
Assume that there is an $s \in S$ which admits a long cycle -- that is a cycle of length at least 3 -- on $U^{k}_{i}$ for some $i$. In this first case we define $k'$ such that $H_{k'}$ acts transitively on all the $V^{k}_i$ with probability $1 - \frac{p}{2}$. Then 

\[\P [S \subseteq H_{k'} \textrm{ and $H_{k'}$ is transitive on the $V^{k}_i$}] \geq \frac{p}{2} >0.\]

If $S \subseteq H_{k'}$ then the $V^{k}_i$ are realized as the $U^{k}_i$. Now we enlist all subsets $S'_l$ in $\Gamma_{k'}$ that contain $S$ and act transitively on the $U^{k}_i$. There are finitely many, so we can find some $S'$ with

\[\P [S' \subseteq H_{k'}] \geq p' >0.\]

We can assume $S'$ to be a subgroup, and by having $s \in S' $ we will show that long cycles of $S'$ cover $U^{k}_{i}$. Indeed, by conjugating $s$ one can move the cycle around in $U^{k}_{i}$, and by the transitivity of $S'$  we get that the whole of $U^{k}_{i}$ is covered. This in turn implies that  long cycles of $S'$ cover $U^{k'}_{i}$ as well.

If on the other hand $S$ acts on $U^{k}_{i}$ by involutions, we will increase $k$ and $S$ while keeping $p$ positive such that the first case holds.

Let $r=\left \lceil \frac{2}{p} \right \rceil$. Furthermore let $k' > k$ such that the shadow of a vertex $v \in U^{k}_{i}$ on $\mathcal{L}_{k'}$ contains at least $r$ vertices, namely $\{v_1, v_2, \ldots, v_r, \ldots\} \subseteq U^{k'}_{i}$. Let $\gamma_1,\gamma_2, \ldots, \gamma_{r} \in \Gamma_{k'+t}$ such that $\gamma_i\in\Rst_{\Gamma_{k'+t}}(v_i)$ and $\gamma_i$ has order at least 3 by Lemma \ref{lemma:high_order_element}, i.e. $\gamma_i$ has a long cycle on $\mathcal{L}_{k'+t}$.

As $H$ is an IRS we have \[\P [S^{\gamma_j} \subseteq H] = \P [S \subseteq H] = p.\]

By Lemma \ref{lemma:intersection_probability} we can find some $j,l$ such that \[\P \big[(S^{\gamma_j} \cup S^{\gamma_l}) \subseteq H\big] \geq \frac{p^3}{6}.\]

Set $S' = \langle S^{\gamma_j} \cup S^{\gamma_l} \rangle$. Pick some $s \in S$ which moves $v \in \mathcal{L}_{k}$. It is easy to check that $s^{\gamma_j} \cdot (s^{\gamma_l})^{-1}  \in S'\cap \Rst_{\Gamma_{k'+t}}(\mathcal{L}_{k})$ has nontrivial sections only at $v_j$, $v_l$, $v_j^{s}$ and $v_{l}^s$, and these sections are some conjugates of $\gamma_j$ and $\gamma_l$, and therefore $s^{\gamma_j} \cdot (s^{\gamma_l})^{-1}$ has a long cycle on $\mathcal{L}_{k'+t}$. So replacing $S$ by $S'$, $k$ by $k'+t$ and $p$ by $\frac{p^3}{6}$ we get to the first case.

Repeating the argument for the first case at most $j$ times we make sure that all $U_{k}^i$ are covered by long cycles, which finishes the proof.
\end{proof}

 \begin{remark}
For the proof of Theorem \ref{theorem:branch_group_closure} one can modify this proof such that instead of $\Gamma_n$, $H_k=\Gamma_n\cap H$ and $S\subseteq \Gamma_n$ we use $G_n=\pi_n(\Gamma)$, $H_n=\pi_n(H)$ and $\pi_n(S)\subseteq G_n$. Another difference is that $H_n$ automatically acts transitively on all the $V_i^n$, so there is no need to distinguish between $k_0$ and $k$.
\end{remark}

\bibliographystyle{alpha}
\bibliography{references}

%

\end{document}